\documentclass[12pt,leqno]{article}
\usepackage{amsfonts}
\usepackage{amsmath, dsfont, amsthm, amsfonts, amssymb, color}
\usepackage{mathrsfs}
\usepackage{color}
\usepackage{stmaryrd}
\newtheorem{thm}{Theorem}[section]

\newtheorem{lem}[thm]{Lemma}

\newtheorem{rem}[thm]{Remark}
\theoremstyle{definition}

\newcommand{\scr}[1]{\mathscr #1}
\definecolor{wco}{rgb}{0.5,0.2,0.3}

\numberwithin{equation}{section} \theoremstyle{remark}

\newcommand{\ua}{\uparrow}

\title{{\bf
Long Time Entropy-Cost type Propagation of Chaos}\footnote{Supported in
 part by  National Key R\&D Program of China (No. 2022YFA1006000) and NNSFC (12271398).} }
\author{
{\bf   Xing Huang  }\\
\footnotesize{ Center for Applied Mathematics, Tianjin
University, Tianjin 300072, China}\\
\footnotesize{  xinghuang@tju.edu.cn}}
\begin{document}
\allowdisplaybreaks
\def\R{\mathbb R}  \def\ff{\frac} \def\ss{\sqrt} \def\B{\mathbf
B} \def\W{\mathbb W}
\def\N{\mathbb N} \def\kk{\kappa} \def\m{{\bf m}}
\def\ee{\varepsilon}\def\ddd{D^*}
\def\dd{\delta} \def\DD{\Delta} \def\vv{\varepsilon} \def\rr{\rho}
\def\<{\langle} \def\>{\rangle} \def\GG{\Gamma} \def\gg{\gamma}
  \def\nn{\nabla} \def\pp{\partial} \def\E{\mathbb E}
\def\d{\text{\rm{d}}} \def\bb{\beta} \def\aa{\alpha} \def\D{\scr D}
  \def\si{\sigma} \def\ess{\text{\rm{ess}}}
\def\beg{\begin} \def\beq{\begin{equation}}  \def\F{\scr F}
\def\Ric{\text{\rm{Ric}}} \def\Hess{\text{\rm{Hess}}}
\def\e{\text{\rm{e}}} \def\ua{\underline a} \def\OO{\Omega}  \def\oo{\omega}
 \def\tt{\tilde} \def\Ric{\text{\rm{Ric}}}
\def\cut{\text{\rm{cut}}} \def\P{\mathbb P} \def\ifn{I_n(f^{\bigotimes n})}
\def\C{\scr C}      \def\aaa{\mathbf{r}}     \def\r{r}
\def\gap{\text{\rm{gap}}} \def\prr{\pi_{{\bf m},\varrho}}  \def\r{\mathbf r}
\def\Z{\mathbb Z} \def\vrr{\varrho}
\def\L{\scr L}\def\Tt{\tt} \def\TT{\tt}\def\II{\mathbb I}
\def\i{{\rm in}}\def\Sect{{\rm Sect}}  \def\H{\mathbb H}
\def\M{\scr M}\def\Q{\mathbb Q} \def\texto{\text{o}}
\def\Rank{{\rm Rank}} \def\B{\scr B} \def\i{{\rm i}} \def\HR{\hat{\R}^d}
\def\to{\rightarrow}\def\l{\ell}\def\iint{\int}
\def\EE{\scr E}\def\Cut{{\rm Cut}}
\def\A{\scr A} \def\Lip{{\rm Lip}}
\def\BB{\scr B}\def\Ent{{\rm Ent}}\def\L{\scr L}
\def\R{\mathbb R}  \def\ff{\frac} \def\ss{\sqrt} \def\B{\mathbf
B}
\def\N{\mathbb N} \def\kk{\kappa} \def\m{{\bf m}}
\def\dd{\delta} \def\DD{\Delta} \def\vv{\varepsilon} \def\rr{\rho}
\def\<{\langle} \def\>{\rangle} \def\GG{\Gamma} \def\gg{\gamma}
  \def\nn{\nabla} \def\pp{\partial} \def\E{\mathbb E}
\def\d{\text{\rm{d}}} \def\bb{\beta} \def\aa{\alpha} \def\D{\scr D}
  \def\si{\sigma} \def\ess{\text{\rm{ess}}}
\def\beg{\begin} \def\beq{\begin{equation}}  \def\F{\scr F}
\def\Ric{\text{\rm{Ric}}} \def\Hess{\text{\rm{Hess}}}
\def\e{\text{\rm{e}}} \def\ua{\underline a} \def\OO{\Omega}  \def\oo{\omega}
 \def\tt{\tilde} \def\Ric{\text{\rm{Ric}}}
\def\cut{\text{\rm{cut}}} \def\P{\mathbb P} \def\ifn{I_n(f^{\bigotimes n})}
\def\C{\scr C}      \def\aaa{\mathbf{r}}     \def\r{r}
\def\gap{\text{\rm{gap}}} \def\prr{\pi_{{\bf m},\varrho}}  \def\r{\mathbf r}
\def\Z{\mathbb Z} \def\vrr{\varrho}
\def\L{\scr L}\def\Tt{\tt} \def\TT{\tt}\def\II{\mathbb I}
\def\i{{\rm in}}\def\Sect{{\rm Sect}}  \def\H{\mathbb H}
\def\M{\scr M}\def\Q{\mathbb Q} \def\texto{\text{o}} \def\LL{\Lambda}
\def\Rank{{\rm Rank}} \def\B{\scr B} \def\i{{\rm i}} \def\HR{\hat{\R}^d}
\def\to{\rightarrow}\def\l{\ell}
\def\8{\infty}\def\I{1}\def\U{\scr U} \def\n{{\mathbf n}}
\maketitle

\begin{abstract} Due to the regularization effect of the stochastic noise, the quantitative entropy-cost type propagation of chaos for mean field interacting particle system is proposed. The result shows that the Kac's chaotic property measured in relative entropy at any positive time can only depend on the weaker initial one measured in $L^2$-Wasserstein distance. Moreover, under dissipative assumption, the long time entropy-cost type propagation of chaos can also be captured. The results are also available in path dependent case, where the log-Sobolev inequality for McKean-Vlasov SDEs does not hold.
 \end{abstract}

 \vskip 2cm
\section{Introduction}
In the mean field interacting particle system, where the coefficients depend on the empirical distribution of the particles, as the number of particles goes to infinity, the limit equation of a single particle is distribution dependent stochastic differential equation(SDE), which is also called McKean-Vlasov SDE in the literature due to the work in \cite{McKean}. This limit phenomenon is related to propagation of chaos, which can be viewed as a dynamical version of Kac's chaos, see for instance \cite{McKean67,SZ}. Recall that a family of Polish space $E$-valued exchangeable random variables $\{Y^{i,N}\}_{1\leq i\leq N}$ is called $\mu$-chaotic for some $\mu\in\scr P(E)$, the space  of all  probability measures on $E$ equipped with the weak topology, if one of the following equivalent statements holds:
\begin{enumerate}
\item[(i)] The $\scr P(E)$-valued random variable $\frac{1}{N}\sum_{i=1}^N\delta_{Y^{i,N}}$ converges to $\mu$ weakly as $N\to\infty$.
\item[(ii)] The joint distribution of $(Y^{1,N},Y^{2,N})$ converges to $\mu^{\otimes2}$ weakly as $N\to\infty$.
\item[(iii)] For any $k\geq 2$, the joint distribution of $(Y^{1,N},Y^{2,N},\cdots,Y^{k,N})$ converges to $\mu^{\otimes k}$ weakly as $N\to\infty$,
\end{enumerate}
here $\mu^{\otimes k}$ denote the $k$ independent product of $\mu$, i.e. $\mu^{\otimes k}=\prod_{i=1}^k\mu$.

Let $(E,\rho)$ be a Polish space and $o$ be a fixed point in $E$. Let
$$\scr P_2(E):=\big\{\mu\in \scr P(E): \mu(\rho(o,\cdot)^2)<\infty\big\},$$
which is a Polish space under the $L^2$-Wasserstein distance
$$\W_2(\mu,\nu)= \inf_{\pi\in \mathbf{C}(\mu,\nu)} \bigg(\int_{E\times E} \rho(x,y)^2 \pi(\d x,\d y)\bigg)^{\ff 1 {2}},\ \  \mu,\nu\in \scr P_2(E), $$ where $\mathbf{C}(\mu,\nu)$ is the set of all couplings of $\mu$ and $\nu$.
%We will also use the total variation distance:
%$$\|\gamma-\tilde{\gamma}\|_{var}=\sup_{\|f\|_\infty\leq 1}|\gamma(f)-\tilde{\gamma}(f)|,\ \ \gamma,\tilde{\gamma}\in \scr P(E).$$
%When $E=\R^m,p\in(0,1]$,
%  the following dual formula holds:
%$$\W_p(\gamma,\tilde{\gamma})=\sup_{[f]_p\leq 1}|\gamma(f)-\tilde{\gamma}(f)|,\ \ \gamma,\tilde{\gamma}\in \scr P_p(\R^m),$$
%where $[f]_p$ denotes the $p$-ordered H\"{o}lder continuous modulus of $f:\R^m\rightarrow\R$ which is defined by $[f]_p:=\sup_{x\neq y}\frac{|f(x)-f(y)|}{|x-y|^p}$ and $|\cdot|$ is the usual Euclidean distance.

Let $\{W_t\}_{t\geq 0}$ be an $n$-dimensional Brownian motion on some complete filtration probability space $(\Omega, \scr F, (\scr F_t)_{t\geq 0},\P)$. Consider McKean-Vlasov SDEs:
\begin{align*}\d  X_t=b_t(X_t,\L_{X_t})\mathrm{d} t+\sigma_t(X_t,\L_{X_t})\mathrm{d} W_t,
\end{align*}
where $\L_{X_t}$ is the distribution of $X_t$, $b:[0,\infty)\times \R^d\times\scr P(\R^d)\to\R^d$, $\sigma:[0,\infty)\times \R^d\times\scr P(\R^d)\to\R^d\otimes\R^{n}$ are measurable and bounded on bounded set.

Let $X_0$ be an $\F_0$-measurable random variable,
$N\ge1$ be an integer and $(X_0^i,W^i_t)_{1\le i\le N}$ be i.i.d.\,copies of $(X_0,W_t).$ Consider independent McKean-Vlasov SDEs:
\begin{align*}\d X_t^i= b_t(X_t^i, \L_{X_t^i})\d t+  \sigma_t(X^i_t,\L_{X_t^i}) \d W^i_t,\ \ 1\leq i\leq N,
\end{align*}
and the mean field interacting particle system
\begin{align}\label{GPS00}\d X^{i,N}_t=b_t(X_t^{i,N}, \hat\mu_t^N)\d t+\sigma_t(X^{i,N}_t, \hat\mu_t^N) \d W^i_t,\ \ 1\leq i\leq N,
\end{align}
where $\hat\mu_t^N$ is the empirical distribution of $(X_t^{i,N})_{1\leq i\leq N}$, i.e.
\begin{align*}
 \hat\mu_t^N =\ff{1}{N}\sum_{j=1}^N\dd_{X_t^{j,N}}.
 \end{align*}
Throughout the paper, we assume that the distribution of the initial values $(X_0^{i,N})_{1\leq i\leq N}$ of the mean field interacting system \eqref{GPS00} is exchangeable.

When $\sigma=I_{d\times d}$ and the interaction is singular, the entropy method is introduced in \cite{BJW,JW,JW1} to derive the quantitative propagation of chaos in relative entropy. More precisely, they deduce
\begin{align}\label{enn}\nonumber&\mathrm{Ent}(\L_{(X_{t}^{1,N},X_{t}^{2,N},\cdots, X_{t}^{k,N})}|\L_{(X_{t}^{1},X_{t}^{2},\cdots,X_{t}^{k})})\\
&\leq \frac{k\mathrm{Ent}(\L_{(X_0^{1,N},X_0^{2,N},\cdots,X_0^{N,N})}|\L_{(X_0^1,X_0^2,\cdots, X_0^N)})}{N}+\frac{C(t)k}{N}, \ \ 1\leq k\leq N, t\in[0,T]
\end{align}
for some constant $C(t)\geq 0$, where the relative entropy of two probability measures is defined as
$$\mathrm{Ent}(\nu|\mu)=\left\{
  \begin{array}{ll}
    \nu(\log(\frac{\d \nu}{\d \mu})), & \hbox{$\nu\ll\mu$;} \\
    \infty, & \hbox{otherwise.}
  \end{array}
\right.$$
They firstly derive global estimate $\mathrm{Ent}(\L_{(X_{t}^{1,N},X_{t}^{2,N},\cdots, X_{t}^{N,N})}|\L_{(X_{t}^{1},X_{t}^{2},\cdots,X_{t}^{N})})$ from the Fokker-Planck equations and then obtain local estimate $\mathrm{Ent}(\L_{(X_{t}^{1,N},X_{t}^{2,N},\cdots, X_{t}^{k,N})}|\L_{(X_{t}^{1},X_{t}^{2},\cdots,X_{t}^{k})})$ for $1\leq k\leq N$ by the tensor property (also called sub-additivity) of relative entropy.

Recently, by directly deducing local estimate $\mathrm{Ent}(\L_{(X_{t}^{1,N},X_{t}^{2,N},\cdots, X_{t}^{k,N})}|\L_{(X_{t}^{1},X_{t}^{2},\cdots,X_{t}^{k})})$ for $1\leq k\leq N$ and combining with the so called BBGKY hierarchy, the author in \cite{L21} obtained the sharp rate $\frac{k^2}{N^2}$ of propagation of chaos in relative entropy for some models such as bounded or uniformly continuous
interaction.

Let us return to \eqref{enn}. To derive the propagation of chaos, the initial relative entropy $\mathrm{Ent}(\L_{(X_0^{1,N},X_0^{2,N},\cdots,X_0^{N,N})}|\L_{(X_0^1,X_0^2,\cdots, X_0^N)})$ is required to be finite, which excludes the example that $X_0^i=x\in\R^d, 1\leq i\leq N$ and $X_0^{i,N}=x+\frac{1}{\sqrt{N}}, 1\leq i\leq N$, where
$\mathrm{Ent}(\L_{(X_0^{1,N},X_0^{2,N},\cdots,X_0^{N,N})}|\L_{(X_0^1,X_0^2,\cdots, X_0^N)})$ is infinite, which means that the initial Kac's chaotic property in relative entropy dose not hold. However, in this example, it is easy to see that
$\W_2(\L_{(X_0^{1,N},X_0^{2,N},\cdots,X_0^{N,N})}, \L_{(X_0^1,X_0^2,\cdots, X_0^N)})=1$.

Let $P_t^\ast\mu_0$ be the distribution of $X_t^1$ with initial distribution $\mu_0\in\scr P_2(\R^d)$. Recall in \cite{W18}, the author has derived the log-Harnack inequality
\begin{align}\label{login}\nonumber&(P_t^\ast \mu_0)(\log f)\\
&\leq\log (P_t^\ast \nu_0)(f)+\frac{c}{t}\W_2(\mu_0,\nu_0)^2,\ \ \mu_0,\nu_0\in\scr P_2(\R^d), f>0, f\in\scr B_b(\R^d), t\in(0,T]
\end{align}
for some constant $c>0$.
Note that \eqref{login} is equivalent to the entropy-cost estimate
\begin{align*}\mathrm{Ent}(P_t^\ast \mu_0|P_t^\ast \nu_0)\leq\frac{c}{t}\W_2(\mu_0,\nu_0)^2,\ \ \mu_0,\nu_0\in\scr P_2(\R^d), t\in(0,T],
\end{align*}
see for instance \cite[Theorem 1.4.2(1)]{Wbook}.
This motivates us to propose a new type of propagation of chaos (called entropy-cost type propagation of chaos), which can be formulated as
\begin{align}\label{pocf}
\nonumber&\mathrm{Ent}(\L_{(X_{t}^{1,N},X_{t}^{2,N},\cdots, X_{t}^{k,N})}|\L_{(X_{t}^{1},X_{t}^{2},\cdots,X_{t}^{k})})\\
&\leq C_0k\psi(N)\Phi(t)\W_2(\L_{(X_0^{i,N})_{1\leq i\leq N}},\L_{(X_0^i)_{1\leq i\leq N}})^2+ C_0k\psi(N), \ \ 1\leq k\leq N, t\in(0,T]
\end{align}
for some constant $C_0\geq 0$ depending on $T$, some decreasing function $\psi:\mathbb{N}^+\to[0,\infty)$ with $\lim_{N\to\infty}\psi(N)=0$, some measurable function $\Phi:(0,\infty)\to(0,\infty)$ with $\lim_{t\to0}\Phi(t)=\infty$.

The uniform in time propagation of chaos also attracts much attention since it is related to generation of chaos, see for instance \cite{Luk,RS}. Generation of chaos means that the larger the time goes, the more chaotic the mean field interacting particle system becomes even if Kac's chaotic property does not hold in the initial time. By using the method of asymptotic reflecting coupling, \cite{DEGZ,GBMEJP} obtain the uniform in time propagation of chaos in Wasserstein distance, see also \cite[Theorem 2.11]{LWZ} for a more explicit and neat result in $L^1$-Wasserstein distance. The uniform in time propagation of chaos in relative entropy is derived in \cite{GBM,M,MRW} by using the uniform in time log-Sobolev inequality for $\L_{X_t^1}$. Combining this with the BBGKY hierarchy, the sharp rate $\frac{k^2}{N^2}$ of uniform in time propagation of chaos in relative entropy is achieved in \cite{LL}. In \cite{RS}, the authors adopt the modulated log-Sobolev inequality to derive the generation of chaos in the sense of modulated free energy and the result is applied in one dimensional repulsive Riesz interaction case with uniformly convex confinement.

The long time entropy-cost type propagation of chaos is given by
\begin{align}\label{lpoc}\nonumber&\mathrm{Ent}(\L_{(X_{t}^{i,N})_{1\leq i\leq k}}|\L_{(X_{t}^{i})_{1\leq i\leq k}})\\
&\leq ck \psi(N)\phi(t)\W_2(\L_{(X_{0}^{i,N})_{1\leq i\leq N}},\L_{(X_{0}^{i})_{1\leq i\leq N}})^2+ck\psi(N),\ \ t>0, 1\leq k\leq N
\end{align}
with some constant $c$ independent of $t$, some decreasing function $\psi:\mathbb{N}^+\to[0,\infty)$ with $\lim_{N\to\infty}\psi(N)=0$, and some measurable function $\phi:(0,\infty)\to(0,\infty)$ with $\lim_{t\to\infty}\phi(t)=0$ and $\lim_{t\to0}\phi(t)=\infty$. Compared with \eqref{pocf}, \eqref{lpoc} can deduce Kac's chaotic property at infinite time, i.e. Kac's chaotic property for ergodic invariant probability measure.
We will establish \eqref{lpoc} in several degenerate models: path dependent and kinetic mean field interacting system with dissipative drifts, where the trick of uniform in time log-Sobolev inequality for $\L_{X_t^1}$ aforementioned is unavailable. In fact, these two models are degenerate ones, where the log-Sobolev inequality does not hold in general. Instead of the entropy method in \cite{BJW,JW,JW1}, we will adopt coupling by change of measure and the dimension free Harnack inequality  as well as an entropy inequality in \cite[Lemma 2.1]{23RW}.

%When \eqref{E0} is well-posed, we denote $X_t^\gamma$ the solution to it from initial distribution $\gamma\in\scr P_2(\R^d)$.
%and let $P_t^\ast\gamma=\L_{X_t^\gamma}$ as well as
%$$P_tf(\gamma):=\E f (X_t^\gamma)=\int_{\R^d}f(x)( P_t^\ast\gamma)(\d x),\ \ f\in\scr B_b(\R^d).$$

%So, we have
%\begin{align*}
%|X_t-Y_t|^2&=|X_0-Y_0|^2+2\int_0^t\<b_s(X_s,\L_{X_s|\F_s^B})-b_s(Y_s,\L_{Y_s|\F_s^B}),X_s-Y_s\>\d s\\
%&\leq |X_0-Y_0|^2+c\int_0^t\left(\W_2(\L_{X_s|\F_s^B}, \L_{Y_s|\F_s^B})^2+|X_s-Y_s|^2\right)\d s.
%\end{align*}

The remaining of the paper is organized as follows: In section 2, a general result on the long time quantitative entropy-cost type  propagation of chaos is offered. The results are applied for path independent case with multiplicative noise in Section 3 and for some degenerate models with additive noise including path dependent and kinetic system in Section 4.

%\subsection{Entropy and total variation distance}
%\begin{align}\label{IIT}\d \hat{X}_t^i= b_t(\hat{X}_t^i+\eta_t, \mu_t^i)\d t+  \sigma_t(\hat{X}^i_t+\eta_t) \d W^i_t,\ \ t\in [0,T],\ \ 1\leq i\leq N
%\end{align}
%and the stochastic interacting particle system
%\begin{align}\label{ITR}\d X^{i,N}_t=b_t(X_t^{i,N}, \hat\mu_t^N)\d t+\sigma_t(X^{i,N}_t) \d W^i_t+ \tt \si_t\d  B_t,\ \ 1\leq i\leq N, X_0^{i,N}=X_0^i,
%\end{align}
%where  $\hat\mu_t^N$ is the empirical distribution of $X_t^{1,N},\cdots,X_t^{N,N}$, i.e.
%\begin{equation*}
% \hat\mu_t^N =\ff{1}{N}\sum_{j=1}^N\dd_{X_t^{j,N}}.
% \end{equation*}
%The diffusion cannot depend on the distribution variable. We need to estimate the conditional entropy given $B_t$, the coefficients before $B$ only depend on $t$. In this case, given $B$, they are independent. Then
%$$\mathrm{Ent}(\L_{X^{i,N}}|B,\L_{X^{i}}|B)\leq \frac{2}{N}\mathrm{Ent}(\L_{X^{N}}|B,\L_{X}|B)$$
\section{A general result}

In this part, we give a general result of the long time entropy-cost type propagation of chaos by using an entropy-cost estimate on finite time along with uniform in time propagation of chaos in $L^2$-Wasserstein distance.

 Let $(E,\rho)$ be a Polish space. Recall that $\scr P(E)$ is the space of all probability measures on $E$ equipped with the weak topology. For any $k\geq 1$,
define the associated Wasserstein distance on $\scr P_2(E^k)$ :
 $$\W_2(\mu,\nu):= \inf_{\pi\in \mathbf{C}(\mu,\nu)} \bigg(\int_{E^k\times E^k} \sum_{i=1}^k\rho(\xi^i,\tilde{\xi}^i)^2\pi(\d \xi,\d \tilde{\xi})\bigg)^{\ff 1 {2}},\ \ \mu,\nu\in \scr P_{2}(E^k),$$
 where $$\xi=(\xi^1,\xi^2,\cdots,\xi^k), \tilde{\xi}=(\tilde{\xi}^1,\tilde{\xi}^2,\cdots,\tilde{\xi}^k)\in E^k. $$
 For any $N\geq 1$, $(\mathbf{X}_t^{i,N})_{1\leq i\leq N}$  is an $E^N$-valued time homogeneous and continuous Markov process with $\L_{\mathbf{X}_t^{i,N}}\in\scr P_2(E)$ and the distribution of $(\mathbf{X}_t^{i,N})_{1\leq i\leq N}$ is exchangeable.  $\{\mathbf{P}^\ast_t\}_{t\geq 0}$ is a family of mappings on $\scr P_2(E)$ which satisfies
\begin{align}\label{SEM}\mathbf{P}_0^\ast\mu=\mu,\ \ \mathbf{P}_s^\ast\mathbf{ P}_t^\ast\mu=\mathbf{P}_{s+t}^\ast \mu,\ \ \mu\in\scr P_2(E).
\end{align}
Recall that for any $\mu\in\scr P(E)$, $k\geq 2$, $\mu^{\otimes k}=\prod_{i=1}^k\mu$ stands for the $k$-independent product of $\mu$.

\begin{thm}\label{GRS} Assume that there exist constants $t_0>r_0\geq 0, c>0$ and continuous functions $g:[0,\infty)\to(0,\infty), h:(r_0,t_0]\to(0,\infty)$ with $\lim_{t\to\infty} g(t)=0, \lim_{t\to r_0} h(t)=\infty$ such that for any $N\geq 1$ and $\mu\in\scr P_2(E)$,
\begin{align}\label{EWE}\nonumber&\mathrm{Ent}(\L_{(\mathbf{X}_{t}^{i,N})_{1\leq i\leq N}}|(\mathbf{P}_{t}^\ast\mu)^{\otimes N})\\
&\leq h(t)\W_2(\L_{(\mathbf{X}_{0}^{i,N})_{1\leq i\leq N}},\mu^{\otimes N})^2+c(1+\mu(\rho(o,\cdot)^2)),\ \ t\in(r_0,t_0],
\end{align}
\begin{align}\label{EWE13}
&\W_2(\L_{(\mathbf{X}_{t}^{i,N})_{1\leq i\leq N}},(\mathbf{P}_{t}^\ast\mu)^{\otimes N})^2\\
\nonumber&\leq g(t)\W_2(\L_{(\mathbf{X}_{0}^{i,N})_{1\leq i\leq N}},\mu^{\otimes N})^2+c(1+\mu(\rho(o,\cdot)^2)), \ \ t\geq 0,
\end{align}
and
\begin{align}\label{uem}\sup_{t\geq 0}(\mathbf{P}_{t}^\ast\mu)(1+\rho(o,\cdot)^2)\leq c(1+\mu(\rho(o,\cdot)^2)).
\end{align}
Then the following assertions hold.
\begin{enumerate}
\item[(i)] There exists a constant $\tilde{c}>0$ such that for any $1\leq k\leq N<\infty$ and $\mu\in\scr P_2(E)$, it holds
\begin{align*}\mathrm{Ent}(\L_{(\mathbf{X}_{t}^{i,N})_{1\leq i\leq k}}|(\mathbf{P}_{t}^\ast\mu)^{\otimes k})
&\leq \tilde{c}\frac{k}{N}g((t-t_0)\vee 0)h(t\wedge t_0)\W_2(\L_{(\mathbf{X}_{0}^{i,N})_{1\leq i\leq N}},\mu^{\otimes N})^2\\
&+\tilde{c}(1+\mu(\rho(o,\cdot)^2))\frac{k}{N}, \ \ t>r_0.
\end{align*}
\item[(ii)] If $\mathbf{P}_t^\ast$ has a unique invariant probability measure $\mu^\ast\in\scr P_2(E)$, then for $1\leq k\leq N<\infty$, it holds
\begin{align*}\mathrm{Ent}(\L_{(\mathbf{X}_{t}^{i,N})_{1\leq i\leq k}}|(\mu^\ast)^{\otimes k})
&\leq \tilde{c}\frac{k}{N}g((t-t_0)\vee 0)h(t\wedge t_0)\W_2(\L_{(\mathbf{X}_{0}^{i,N})_{1\leq i\leq N}},(\mu^\ast)^{\otimes N})^2\\
&+\tilde{c}(1+\mu^\ast(\rho(o,\cdot)^2))\frac{k}{N},\ \ t>r_0.
\end{align*}
%and $$\W_2(\L_{\mathbf{X}_{t}^{1,N}},\mu^\ast)^2\leq r(t)\W_2(\L_{\mathbf{X}_{0}^{1,N}},\mu^\ast)^2+\frac{c}{N}.$$
\item[(iii)] If in addition, for any $N\geq 1$, there exists $\bar{\mu}^N\in \scr P_2(E^N)$ such that
$$\lim_{t\to \infty}\W_2(\L_{(\mathbf{X}_{t}^{i,N})_{1\leq i\leq N}},\bar{\mu}^N)=0,$$
then for $1\leq k\leq N<\infty$, we have
\begin{align*}\mathrm{Ent}(\bar{\mu}^N\circ \pi_k^{-1}|(\mu^\ast)^{\otimes k})\leq \tilde{c}(1+\mu^\ast(\rho(o,\cdot)^2))\frac{k}{N},
\end{align*}
%and \begin{align*}\W_2(\bar{\mu}^N\circ \pi_1^{-1}|\mu^\ast)\leq c\frac{1}{N},
%\end{align*}
where $\pi_k$ is the projection mapping form $E^N$ to $E^k$.
\end{enumerate}
\end{thm}
\begin{proof}
(i)
By the Markovian  property of $(X_t^{i,N})_{1\leq i\leq N}$, \eqref{SEM}-\eqref{uem}, we derive for any $t\geq t_0$ and $\mu\in\scr P_2(E)$,
\begin{align*}&\mathrm{Ent}(\L_{(\mathbf{X}_{t}^{i,N})_{1\leq i\leq N}}|(\mathbf{P}_{t}^\ast\mu)^{\otimes N})\\
&\leq h(t_0)\W_2(\L_{(\mathbf{X}_{t-t_0}^{i,N})_{1\leq i\leq N}},(\mathbf{P}_{t-t_0}^\ast\mu)^{\otimes N})^2 +c(1+(\mathbf{P}_{t-t_0}^\ast\mu)(\rho(o,\cdot)^2))\\ &\leq h(t_0)g(t-t_0)\W_2(\L_{(\mathbf{X}_{0}^{i,N})_{1\leq i\leq N}},\mu^{\otimes N})^2+(c^2+c)(1+\mu(\rho(o,\cdot)^2)).
\end{align*}
Combining this with \eqref{EWE}, the sub-additivity (also called tensor property) of relative entropy
$$\mathrm{Ent}(\L_{(\mathbf{X}_{t}^{i,N})_{1\leq i\leq k}}|(\mathbf{P}_{t}^\ast\mu)^{\otimes k})\leq \frac{2k}{N}\mathrm{Ent}(\L_{(\mathbf{X}_{t}^{i,N})_{1\leq i\leq N}}|(\mathbf{P}_{t}^\ast\mu)^{\otimes N}), \ \ 1\leq k\leq N,$$
we complete the proof of (i). (ii) follows by taking $\mu=\mu^\ast$ in (i). Finally, letting $t\to\infty$ in (ii), we get (iii) by the lower-semicontinuity of relative entropy.
\end{proof}
\section{Applications in path independent case with multiplicative noise}
Let $b: \R^d\times\scr P(\R^d)\to\R^d$, $\sigma:\R^d\times \scr P(\R^d)\to\R^d\otimes\R^{n}$ be measurable, and $W_t$ and $W_t^i$ be defined in Section 1. In this section, we consider
\begin{align}\label{Eb3}\d  X_t=b(X_t,\L_{X_t})\mathrm{d} t+\sigma(X_t,\L_{X_t})\mathrm{d} W_t.
\end{align}
Correspondingly, we investigate
\begin{align*}\d  X_t^i=b(X_t^i,\L_{X_t^i})\mathrm{d} t+\sigma(X_t^i,\L_{X_t^i})\mathrm{d} W_t^i,\ \ 1\leq i\leq N,
\end{align*}
and the mean field interacting particle system:
\begin{align*}\d X^{i,N}_t=b(X_t^{i,N}, \ff{1}{N}\sum_{j=1}^N\dd_{X_t^{j,N}})\d t+\sigma(X^{i,N}_t,\ff{1}{N}\sum_{j=1}^N\dd_{X_t^{j,N}}) \d W^i_t,\ \ 1\leq i\leq N.
\end{align*}
We make the following assumptions.
\begin{enumerate}
\item[{\bf(H)}] $b(x,\mu)=b^{(0)}(x)+\int_{\R^d}b^{(1)}(x,y)\mu(\d y)$, $\sigma(x,\mu)=\int_{\R^d}\tilde{\sigma}(x,y)\mu(\d y)$ and $b^{(0)}$ is continuous. There exist constants $K_1, K_2>0$ with $K_1>8K_2$ such that
\begin{align}\label{und}
&2\<b^{(0)}(x)-b^{(0)}(\tilde{x}),x-\tilde{x}\>\leq -K_1|x-\tilde{x}|^2,\ \   x,\tilde{x}\in\R^d,
\end{align}
\begin{align*}
&\|\tilde{\sigma}(x,y)-\tilde{\sigma}(\tilde{x},\tilde{y})\|_{HS}^2\leq K_2(|x-\tilde{x}|^2+|y-\tilde{y}|^2),\ \   x,\tilde{x},y,\tilde{y}\in\R^d,
\end{align*}
and
$$|b^{(1)}(x,y)-b^{(1)}(\tilde{x},\tilde{y})|\leq K_2(|x-\tilde{x}|+|y-\tilde{y}|),\ \ x,\tilde{x},y,\tilde{y}\in\R^d.$$
\end{enumerate}
The main result in this part is the following theorem.
\begin{thm}\label{POC30}
Assume {\bf(H)} and $\L_{X_0^{1,N}},\L_{X_0^{1}}\in \scr P_2(\R^d)$.
Then the following assertions hold.

(1) There exists a constant $c>0$ such that
\begin{equation}\begin{split}\label{S13}
\E\sum_{i=1}^N|X^i_s-X^{i,N}_s|^2&\leq \e^{-\frac{K_1-8K_2}{2}s}\E\sum_{i=1}^N|X_0^{i,N}-X_0^{i}|^2\\
&+c\left[\frac{4}{(K_1-8K_2)^2}+\frac{4}{K_1-8K_2}\right](1+\E|X_0^1|^{2}),\ \ s\geq 0.
\end{split}\end{equation}
%Consequently,
%\begin{equation}\begin{split}\label{UIP}
%&\W_2(\bar{\mu}^N\circ\pi_1^{-1},\mu^\ast)^2\leq \frac{4}{(K_1-4K_2)^2}\frac{c}{N}.
%\end{split}\end{equation}

(2) If $\sigma(x,\mu)$ does not depend on $\mu$ and $\delta^{-1}\leq\sigma\sigma^\ast\leq \delta$ holds for some $\delta\geq 1$, then there exists a constant $c>0$ such that
\begin{align*}& \mathrm{Ent}(\L_{(X_{t}^{1,N},X_{t}^{2,N},\cdots, X_{t}^{k,N})}|\L_{(X_{t}^{1},X_{t}^{2},\cdots,X_{t}^{k})})\\
&\leq c\frac{k}{N}\frac{\e^{-\frac{K_1-8K_2}{2}t}}{t\wedge1}\W_2(\L_{(X_{0}^{i,N})_{1\leq i\leq N}},\L_{(X_{0}^{i})_{1\leq i\leq N}})^2+ck\frac{1+\E|X_0^1|^{2}}{N},\ \ 1\leq k\leq N,t>0.
\end{align*}
Consequently, it holds
\begin{align*}&\mathrm{Ent}(\L_{(X_{t}^{1,N},X_{t}^{2,N},\cdots, X_{t}^{k,N})}|(\mu^\ast)^{\otimes k})\\
&\leq c\frac{k}{N}\frac{\e^{-\frac{K_1-8K_2}{2}t}}{t\wedge1}\W_2(\L_{(X_{0}^{i,N})_{1\leq i\leq N}},(\mu^\ast)^{\otimes N})^2+ck\frac{1+\mu^\ast(|\cdot|^2)}{N},\ \ 1\leq k\leq N,t>0,
\end{align*}
and
\begin{align*}&\mathrm{Ent}(\bar{\mu}^N\circ\pi_k^{-1}|(\mu^\ast)^{\otimes k})+\W_2(\bar{\mu}^N\circ\pi_k^{-1},(\mu^\ast)^{\otimes k})^2\leq ck\frac{1+\mu^\ast(|\cdot|^2)}{N},\ \ 1\leq k\leq N,
\end{align*}
where $\pi_k$ is the projecting mapping from $(\R^d)^N$ to $(\R^d)^k$, $\bar{\mu}^N$ and $\mu^\ast$ are the unique invariant probability measures of $(X_t^{1,N}, X_t^{2,N},\cdots,X_t^{N,N})$ and $X_t^1$ respectively.
%(2) If in particular, $b_t(x,\mu)=\int_{\R^d}\tilde{b}(x-y)\mu(\d x)$ for some Lipschitz continuous function  $\tilde{b}$, then for any $X_0^i$ with $\L_{X_0^i}\in\scr P_2$, the assertion in (1) holds for
%$\frac{1}{N}$ replacing $R_{d,q}(N)$.
\end{thm}
\begin{rem}\label{myt} Compared with \cite{JW} for the quantitative propagation of chaos in relative entropy, in Theorem \ref{POC30}(2), to derive the Kac's chaotic property measured in relative entropy at any positive time, it is required that
$$\lim_{N\to\infty}\frac{\W_2(\L_{(X_0^{i,N})_{1\leq i\leq N}},\L_{(X_0^i)_{1\leq i\leq N}})^2}{N}=0$$ instead of
$$\lim_{N\to\infty}\frac{\mathrm{Ent}(\L_{(X_0^{1,N},X_0^{2,N},\cdots,X_0^{N,N})}|\L_{(X_0^1,X_0^2,\cdots, X_0^N)})}{N}=0.$$
The present result allows $\L_{(X_0^{1,N},X_0^{2,N},\cdots,X_0^{N,N})}$ to be singular with $\L_{(X_0^1,X_0^2,\cdots, X_0^N)}$.
\end{rem}
\begin{rem} To derive the uniform in time propagation of chaos in $L^1$-Wasserstein distance, one may adopt the asymptotic reflection coupling and only assume the partially dissipative assumption on $b^{(0)}$, i.e.
\begin{align}\label{par}
&\<b^{(0)}(x)-b^{(0)}(\tilde{x}),x-\tilde{x}\>\leq -K_1|x-\tilde{x}|^21_{\{|x-\tilde{x}|\geq R\}}+K_2|x-\tilde{x}|^21_{\{|x-\tilde{x}|<R\}},\ \   x,\tilde{x}\in\R^d
\end{align}
for some constants $K_1>0, K_2\geq 0, R>0$, see \cite{DEGZ} for the additive noise case and \cite{HX25d} for the case $\sigma$ depending on both spatial and measure variables. However, in the present multiplicative noise case, the partially dissipative assumption \eqref{par} is not sufficient to derive the uniform in time propagation of chaos in $L^2$-Wasserstein distance. This is the reason why we give the uniformly dissipative condition \eqref{und} in Theorem \ref{POC30}. We leave the partially dissipative condition \eqref{par} to derive uniform in time propagation of chaos in $L^2$-Wasserstein distance in the future research.
\end{rem}
Before giving the proof of Theorem \ref{POC30}, we give a lemma on the convergence rate of the law of large number in $L^2(\P)$. Note that when $h$ is $\R^d$ or $\R^{d}\otimes \R^n$-valued, Lemma \ref{CTY} below still holds for $d\tilde{c}$ or $dn\tilde{c}$ replacing $\tilde{c}$ respectively.
\begin{lem}\label{CTY} Let $(V,\|\cdot\|_V)$ be a Banach space. $(Z_i)_{i\geq 1}$ are i.i.d. $V$-valued random variables with $\E\|Z_1\|^2_V<\infty$ and $h:V\times V\to \R$ is  measurable and of at most linear growth, i.e. there exists a constant $c>0$ such that
$$|h(v,\tilde{v})|\leq c(1+\|v\|_V+\|\tilde{v}\|_V),\ \ v,\tilde{v}\in V.$$
Then there exists a constant $\tilde{c}>0$ such that
\begin{align*}\E\left|\frac{1}{N}\sum_{m=1}^N h(Z_1,Z_m)-\int_{V} h(Z_1,y)\L_{Z_1}(\d y)\right|^2\leq \frac{\tilde{c}}{N}\E(1+\|Z_1\|_{V}^2).
\end{align*}
\end{lem}
\begin{proof}
Since $\{Z_i\}_{1\leq i\leq N}$ are i.i.d., it follows that
\begin{align*}&\E\left|\frac{1}{N}\sum_{m=1}^N h(Z_1,Z_m)-\int_{V}  h(Z_1,y)\L_{Z_1}(\d y)\right|^2\\
&\leq 2\E\left|\frac{1}{N} h(Z_1,Z_1)\right|^2+2\E\left|\frac{1}{N}\sum_{m=2}^N h(Z_1,Z_m)-\int_{V} h(Z_1,y)\L_{Z_1}(\d y)\right|^2\\
&= 2\E\left|\frac{1}{N} h(Z_1,Z_1) \right|^2\\
&\qquad\quad+2\E\left\{\left\{\E\left|\frac{1}{N}\sum_{m=2}^N h(x,Z_m)-\int_{V} h(x,y)\L_{Z_1}(\d y)\right|^2\right\}\Bigg|_{x=Z_1}\right\}\\
&= 2\E\left|\frac{1}{N} h(Z_1,Z_1)\right|^2\\
&\qquad\quad+2\E\left\{\left\{\frac{N-1}{N^2} \E| h(x,Z_1)|^2+\frac{2-N}{N^2}(\E | h(x,Z_1)|)^2\right\}\Bigg|_{x=Z_1}\right\}\\
&\leq \frac{\tilde{c}}{N}\E(1+\|Z_1\|_V^2)
\end{align*}
for some constant $\tilde{c}>0$. So, the proof is completed.
\end{proof}
\begin{proof}[Proof of Theorem \ref{POC30}] (1) Firstly, {\bf(H)} implies that
\begin{align}\label{S-S}&\nonumber\|\sigma(x,\mu)-\sigma(\tilde{x},\nu)\|_{HS}^2\\
\nonumber&= \inf_{\pi\in \mathbf{C}(\mu,\nu)}\left\|\int_{\R^d\times\R^d}[\tilde{\sigma}(x,y)-\tilde{\sigma}(\tilde{x},\tilde{y})]\pi(\d y,\d\tilde{y})\right\|_{HS}^2\\
&\leq \inf_{\pi\in \mathbf{C}(\mu,\nu)}\int_{\R^d\times\R^d}\|\tilde{\sigma}(x,y)-\tilde{\sigma}(\tilde{x},\tilde{y})\|_{HS}^2\pi(\d y,\d\tilde{y})\\
\nonumber&\leq K_2(|x-\tilde{x}|^2+\W_2(\mu,\nu)^2),\ \ x,\tilde{x}\in\R^d,\mu,\nu\in\scr P_2(\R^d),
\end{align}
and
\begin{align}\label{b-b}\nonumber&2\left\<\int_{\R^d}b^{(1)}(x,y)\mu(\d y)-\int_{\R^d}b^{(1)}(\tilde{x},\tilde{y})\nu(\d \tilde{y}),x-\tilde{x}\right\>\\
&\leq  2K_2|x-\tilde{x}|^2+2K_2\W_1(\mu,\nu)|x-\tilde{x}|\\
\nonumber&\leq 3K_2|x-\tilde{x}|^2+K_2\W_1(\mu,\nu)^2,\ \ x,\tilde{x}\in\R^d,\mu,\nu\in\scr P_2(\R^d).
\end{align}
By \eqref{S-S} and \eqref{b-b}, we conclude that
\begin{align*}
&2\<b(x,\mu)-b(\tilde{x},\nu),x-\tilde{x}\>+\|\sigma(x,\mu)-\sigma(\tilde{x},\nu)\|_{HS}^2\\
&\leq (-K_1+4K_2)|x-\tilde{x}|^2+2K_2\W_2(\mu,\nu)^2,\ \ x,\tilde{x}\in\R^d,\mu,\nu\in\scr P_2(\R^d).
\end{align*}
This together with $K_1>8K_2$ and \cite[Theorem 3.1]{W18} yields that \eqref{Eb3} has a unique invariant probability measure $\mu^\ast$ satisfying
\begin{align}\label{INV}\W_2(\L_{X_t^1},\mu^\ast)^2\leq \e^{-(K_1-6K_2)t}\W_2(\L_{X_0^1},\mu^\ast)^2, \ \ t\geq 0,
\end{align}
and there exists a constant $C>0$ such that
\begin{align}\label{uni00}
\sup_{t\geq 0}\E(1+|X_t^1|^2)<C(1+\E|X_0^1|^2).
\end{align}
Applying It\^{o}'s formula, we derive from {\bf(H)} that
\begin{align}\label{GYA}
\nonumber&\d|X^{i,N}_t-X^{i}_t|^2\\
\nonumber&\leq -K_1|X^{i,N}_t-X^{i}_t|^2\d t+\d M_t^i\\
&+\left\|\frac{1}{N}\sum_{m=1}^N  \tilde{\sigma}(X_t^{i,N},X_t^{m,N})-\int_{\R^d}\tilde{\sigma}(X_t^{i},y)\L_{X_t^1}(\d y)\right\|_{HS}^2\d t\\
\nonumber&+2\left\<\frac{1}{N}\sum_{m=1}^N  b^{(1)}(X_t^{i,N},X_t^{m,N})-\int_{\R^d} b^{(1)}(X_t^{i},y)\L_{X_t^1}(\d y),X^{i,N}_t-X^{i}_t\right\>\d t
\end{align}
for some martingale $M_t^i$. Observe that
\begin{align}\label{WDD}\W_2(\frac{1}{N}\sum_{i=1}^N\delta_{x_i},\frac{1}{N} \sum_{i=1}^N\delta_{\tilde{x}_i})^2\leq \frac{1}{N}\sum_{i=1}^N|x_i-\tilde{x}_i|^2,\ \ x_i,\tilde{x}_i\in\R^d, 1\leq i\leq N.
\end{align}
Again by \eqref{S-S}, \eqref{b-b} and \eqref{WDD}, we have
\begin{align}\label{b1x}
\nonumber&2\left\<\frac{1}{N}\sum_{m=1}^N  b^{(1)}(X_t^{i,N},X_t^{m,N})-\int_{\R^d} b^{(1)}(X_t^{i},y)\L_{X_t^1}(\d y),X^{i,N}_t-X^{i}_t\right\>\\
\nonumber&\leq 2\left\<\frac{1}{N}\sum_{m=1}^N  b^{(1)}(X_t^{i,N},X_t^{m,N})-\frac{1}{N}\sum_{m=1}^N  b^{(1)}(X_t^{i},X_t^{m}),X^{i,N}_t-X^{i}_t\right\>\\
&+2\left\<\frac{1}{N}\sum_{m=1}^N  b^{(1)}(X_t^{i},X_t^{m})-\int_{\R^d} b^{(1)}(X_t^{i},y)\L_{X_t^1}(\d y),X^{i,N}_t-X^{i}_t\right\>\\
\nonumber&\leq 3K_2|X^{i,N}_t-X^{i}_t|^2+K_2\frac{1}{N}\sum_{m=1}^N |X_t^{m,N}-X_t^{m}| ^2+\frac{K_1-8K_2}{2}|X^{i,N}_t-X^{i}_t|^2\\
\nonumber&+\frac{2}{K_1-8K_2}\left|\frac{1}{N}\sum_{m=1}^Nb^{(1)}(X_t^{i}, X_t^{m})-\int_{\R^d}b^{(1)}(X_t^{i}, y)\L_{X_t^1}(\d y)\right|^2,
\end{align}
and
\begin{align}\label{sigm0}
\nonumber&\left\|\frac{1}{N}\sum_{m=1}^N  \tilde{\sigma}(X_t^{i,N},X_t^{m,N})-\int_{\R^d} \tilde{\sigma}(X_t^{i},y)\L_{X_t^1}(\d y)\right\|_{HS}^2\\
\nonumber&\leq 2\left\|\frac{1}{N}\sum_{m=1}^N  \tilde{\sigma}(X_t^{i,N},X_t^{m,N})-\frac{1}{N}\sum_{m=1}^N  \tilde{\sigma}(X_t^{i},X_t^{m})\right\|_{HS}^2\\
&+2\left\|\frac{1}{N}\sum_{m=1}^N  \tilde{\sigma}(X_t^{i},X_t^{m})-\int_{\R^d} \tilde{\sigma}(X_t^{i},y)\L_{X_t^1}(\d y)\right\|_{HS}^2\\
\nonumber&\leq 2K_2|X^{i,N}_t-X^{i}_t|^2+2K_2\frac{1}{N}\sum_{m=1}^N |X_t^{m,N}-X_t^{m}|^2\\
\nonumber&+2\left\|\frac{1}{N}\sum_{m=1}^N  \tilde{\sigma}(X_t^{i},X_t^{m})-\int_{\R^d} \tilde{\sigma}(X_t^{i},y)\L_{X_t^1}(\d y)\right\|_{HS}^2.
\end{align}
Substituting \eqref{b1x}-\eqref{sigm0} into \eqref{GYA} and using Lemma \ref{CTY} and \eqref{uni00}, we can find a constant $c>0$ such that for any $s\geq 0$,
\begin{align*}
&\E\sum_{i=1}^N|X^{i,N}_s-X^{i}_s|^2\\
&\leq \e^{-\frac{K_1-8K_2}{2}s}\E\sum_{i=1}^N|X^{i,N}_0-X^{i}_0|^2\\
&+\int_0^s\frac{2\e^{-\frac{K_1-8K_2}{2}(s-t)}}{K_1-8K_2}\sum_{i=1}^N\E\left|\frac{1}{N}\sum_{m=1}^Nb^{(1)}(X_t^{i}, X_t^{m})-\int_{\R^d}b^{(1)}(X_t^{i}, y)\L_{X_t^1}(\d y)\right|^2\d t\\
&+\int_0^s2\e^{-\frac{K_1-8K_2}{2}(s-t)}\sum_{i=1}^N\E\left\|\frac{1}{N}\sum_{m=1}^N  \tilde{\sigma}(X_t^{i},X_t^{m})-\int_{\R^d} \tilde{\sigma}(X_t^{i},y)\L_{X_t^1}(\d y)\right\|_{HS}^2\d t\\
&\leq \e^{-\frac{K_1-8K_2}{2}s}\E\sum_{i=1}^N|X^{i,N}_0-X^{i}_0|^2+\int_0^s\frac{2\e^{-\frac{K_1-8K_2}{2}(s-t)}}{K_1-8K_2} c(1+\E|X_0^1|^2)\d t\\
&+\int_0^s2\e^{-\frac{K_1-8K_2}{2}(s-t)}c(1+\E|X_0^1|^2)\d t\\
&\leq  \e^{-\frac{K_1-8K_2}{2}s}\E\sum_{i=1}^N|X^{i,N}_0-X^{i}_0|^2+\left[\frac{4}{(K_1-8K_2)^2} +\frac{4}{K_1-8K_2}\right]c(1+\E|X_0^1|^2).
\end{align*}
Hence, we finish the proof of (1).

(2)  Define
\begin{align*}\P^{0}:= \P(\ \cdot\ |\F_0),\ \ \E^{0}:= \E(\ \cdot\ | \F_0).
\end{align*}
Let $\L_{\xi|\P^0}$ denote the conditional distribution of a random variable $\xi$ with respect to $\F_0$. Consider
\begin{align}\label{GPSer}\d \bar{X}_t^{i}&= b_t^{(0)}(\bar{X}_t^{i})\d t+\int_{\R^d}b_t^{(1)}(\bar{X}_t^{i}, y)\mu_t(\d y)\d t+  \sigma_t(\bar{X}^{i}_t) \d W^i_t,\ \ \bar{X}_0^{i}=X_0^{i,N}, 1\leq i\leq N.
\end{align}

We first assume that $b^{(1)}$ is bounded.
Rewrite \eqref{GPSer} as
\begin{align}\label{MXY}\nonumber  \d \bar{X}_t^{i}&=  b_t^{(0)}(\bar{X}_t^{i})\d t+\frac{1}{N}\sum_{m=1}^Nb_t^{(1)}(\bar{X}_t^{i}, \bar{X}_t^{m})\d t+  \sigma_t(\bar{X}^{i}_t) \d \hat{W}^i_t,\ \ \bar{X}_0^{i}=X_0^{i,N}, 1\leq i\leq N,
\end{align}
where
\begin{equation*}\begin{split}
&\hat{W}_t^i :=   W_t^i-\int_0^t \gamma_s^i\d s,\\ &\gamma_t^i:=[\sigma_t^\ast(\sigma_t\sigma_t^\ast)^{-1}](\bar{X}^{i}_t)\left(\frac{1}{N}\sum_{m=1}^Nb_t^{(1)}(\bar{X}_t^{i}, \bar{X}_t^{m})-\int_{\R^d}b_t^{(1)}(\bar{X}_t^{i}, y)\mu_t(\d y)\right),\ \ 1\leq i\leq N.
\end{split}\end{equation*}
Let $t_1\in(0,1]$ and set
\begin{equation*}\begin{split}
&\hat{W}_t=(\hat{W}_t^1,\hat{W}_t^2,\cdots, \hat{W}_t^N),\\
& R_t:= \e^{\int_0^{t}\sum_{i=1}^N\<\gamma_r^i, \d W^i _r\> -\ff 1 2 \int_0^{t} \sum_{i=1}^N|\gamma_r^i|^2\d r},\\
& \Q_t^{0}:= R_t\P^{0}, \ \ t\in [0,t_1].
\end{split}\end{equation*}
Since $b^{(1)}$ is bounded, Girsanov's theorem implies that
$\{\hat{W}_{t}\}_{t\in[0,t_1]}$ is an $(N\times n)$-dimensional Brownian motion under the weighted conditional probability $\Q^{0}_{t_1}$.
So, we have $\L^{\Q^{0}_{t_1}}_{(\bar{X}^i_t)_{1\leq i\leq N}}=\L^{\P^{0}}_{(X_t^{i,N})_{1\leq i\leq N}},\ \ t\in[0,t_1]$.
Using Young's inequality, for any $0< F\in \B_b((\R^d)^N)$, it holds
\beg{align}\label{DDT} \nonumber&\E^{0} \log F(X_{t_1}^{1,N},X_{t_1}^{2,N},\cdots, X_{t_1}^{N,N})\\
\nonumber&\le \log \E^{0} [ F(\bar{X}_{t_1}^{1},\bar{X}_{t_1}^{2},\cdots,\bar{X}_{t_1}^{N})]\\
 &+ \frac{\delta}{2} \sum_{i=1}^N\E_{\Q^{0}_{t_1}}\int_0^{t_1}\left|\frac{1}{N}\sum_{m=1}^Nb_t^{(1)}(\bar{X}_t^{i}, \bar{X}_t^{m})-\int_{\R^d}b_t^{(1)}(\bar{X}_t^{i}, y)\mu_t(\d y)\right|^2\d t\\
\nonumber&=\log \E^{0} [ F(\bar{X}_{t_1}^{1},\bar{X}_{t_1}^{2},\cdots,\bar{X}_{t_1}^{N})]\\
\nonumber &+ \frac{\delta}{2} \sum_{i=1}^N\E^0\int_0^{t_1}\left|\frac{1}{N}\sum_{m=1}^Nb_t^{(1)}(X_t^{i,N}, X_t^{m,N})-\int_{\R^d}b_t^{(1)}(X_t^{i,N}, y)\mu_t(\d y)\right|^2\d t.
\end{align}
In general, let $b^{(1,(n))}=(-n\vee \<b^{(1)},e_i\>\wedge n)_{1\leq i\leq d}, n\geq 1$ with $(e_i)_{1\leq i\leq d}$ being a standard orthogonal basis in $\R^d$. \eqref{DDT} follows by an approximation technique, \cite[Theorem 1.4.2(2)]{Wbook} and the lower-semicontinuity of relative entropy.

%So, we have
%\beg{align}\label{DDY} \nonumber&\E^{0} \log F(X_{t_0}^{1,N},X_{t_0}^{2,N},\cdots, X_{t_0}^{N,N})\\
%&\le \log \E^{0} [ F(\bar{X}_{t_0}^{1},\bar{X}_{t_0}^{2},\cdots,\bar{X}_{t_0}^{N})]+C+CN\E|X_0^{1,N}-X_0^1|^2, \ \ 0< F\in \B_b((\R^d)^N).\end{align}

Next, let $$b_t^{\mu}(x)=b_t\left(x,\mu_t\right), \ \ t\in[0,1], x\in\R^d.$$
Observe that
\begin{align*}
\d X_t^{i}=b_t^{\mu}(X_t^{i})+\sigma_t(X_t^{i})\d W_t^i,\ \ X_0^i=X_0^i, 1\leq i\leq N,
\end{align*}
and
\begin{align*}
\d \bar{X}_t^{i}=b_t^{\mu}(\bar{X}_t^{i})+\sigma_t(\bar{X}_t^{i})\d W_t^i,\ \ \bar{X}_0^i=X_0^{i,N}, 1\leq i\leq N.
\end{align*}

By \cite[Theorem 3.4.1]{Wbook} and {\bf(H)}, for large enough $p>1$ , we get the Harnack inequality with power $p$:
\begin{align}\label{Har}\nonumber&\left(\E^{0} [F(\bar{X}_t^{1},\bar{X}_t^{2},\cdots,\bar{X}_t^{N})]\right)^p\\
&\leq \E^{0} [F(X_t^{1},X_t^{2},\cdots,X_t^{N})^p]\\
\nonumber&\qquad\quad\times\exp\left\{\frac{c(p)\sum_{i=1}^N|X_0^{i,N}-X_0^{i}|^2}{t}\right\}, \ \ F\in \scr B^+_b((\R^d)^N),t\in(0,1].
\end{align}
In view of \cite[Theorem 1.4.2(1)]{Wbook}, \eqref{Har} is equivalent to
\begin{align}\label{PES}\nonumber&\int_{(\R^{d})^N}\left(\frac{\d \L_{\{\bar{X}_t^{i}\}_{1\leq i\leq N}|\P^0}}{\d \L_{\{X_t^i\}_{1\leq i\leq N}|\P^0}}\right)^{\frac{p}{p-1}}\d \L_{\{X_t^i\}_{1\leq i\leq N}|\P^0}\\
&\leq \exp\left\{\frac{c(p)\sum_{i=1}^N|X_0^{i,N}-X_0^{i}|^2}{(p-1)t}\right\},\ \ t\in(0,1].
\end{align}

\eqref{PES} together with \eqref{DDT}, \cite[Theorem 1.4.2(2)]{Wbook}, \cite[Lemma 2.1]{23RW} implies that
there exists $p>1$ such that
\begin{align*}
\nonumber&\E^{0} \log F(X_{t}^{1,N},X_{t}^{2,N},\cdots, X_{t}^{N,N})\\
\nonumber&\leq  \log \E^{0} [ F(X_{t}^{1},X_{t}^{2},\cdots,X_{t}^{N})]\\
+ &\frac{\delta}{2}p\sum_{i=1}^N\int_0^{t}\E^0\left|\frac{1}{N}\sum_{m=1}^N b_r^{(1)}(X_r^{i,N},X_r^{m,N})-\int_{\R^d} b_r^{(1)}(X_r^{i,N},y)\mu_r(\d y)\right|^2\d r\\
\nonumber&+\frac{c(p)\sum_{i=1}^N|X_0^{i,N}-X_0^{i}|^2}{t}, \ \ 0< F\in \B_b((\R^d)^N), t\in(0,1]
\end{align*}
for some constant $c(p)>0$.
Taking expectation on both sides and using Jensen's inequality and the symmetry of $(X^{i,N})_{1\leq i\leq N}$, we have
\begin{align}\label{ENW}
\nonumber&\E \log F(X_{t}^{1,N},X_{t}^{2,N},\cdots, X_{t}^{N,N})\\
\nonumber&\leq  \log \E [ F(X_{t}^{1},X_{t}^{2},\cdots,X_{t}^{N})]\\
&+ \frac{\delta}{2}p\sum_{i=1}^N\int_0^{t}\E\left|\frac{1}{N}\sum_{m=1}^N b_r^{(1)}(X_r^{i,N},X_r^{m,N})-\int_{\R^d} b_r^{(1)}(X_r^{i,N},y)\L_{X_r^1}(\d y)\right|^2\d r\\
\nonumber&+\frac{c(p)\E\sum_{i=1}^N|X_0^{i,N}-X_0^{i}|^2}{t}.
%\nonumber&\leq  \log \E [ F(X_{t}^{1},X_{t}^{2},\cdots,X_{t}^{N})]+ cpN(\E|X^{1,N}_0-X^{1}_0|^2+\frac{1}{N})\\
%\nonumber&+\frac{c(p)N\E|X_0^{i,N}-X_0^{i}|^2}{t}, \ \ 0< F\in \B_b((\R^d)^N).
\end{align}
Combining Lemma \ref{CTY} with {\bf(H)}, \eqref{S13} and \eqref{uni00}, there exists constants $c, C_0>0$ such that
\begin{align}\label{A4}
\nonumber&\sum_{i=1}^N\E\left|\frac{1}{N}\sum_{m=1}^N  b_t^{(1)}(X_t^{i,N},X_t^{m,N})-\int_{\R^d} b_t^{(1)}(X_t^{i,N},y)\mu_t(\d y)\right|^2\\
\nonumber&\leq 2\sum_{i=1}^N\E\bigg|\frac{1}{N}\sum_{m=1}^N b_t^{(1)}(X_t^{i,N},X_t^{m,N})-\int_{\R^d} b_t^{(1)}(X_t^{i,N},y)\mu_t(\d y)\\
&\qquad\qquad\quad-\bigg(\frac{1}{N}\sum_{m=1}^N  b_t^{(1)}(X_t^{i},X_t^{m})-\int_{\R^d} b_t^{(1)}(X_t^{i},y)\mu_t(\d y)\bigg)\bigg|^2\\
\nonumber&+2\sum_{i=1}^N\E\left|\frac{1}{N}\sum_{m=1}^N  b_t^{(1)}(X_t^{i},X_t^{m})-\int_{\R^d} b_t^{(1)}(X_t^{i},y)\mu_t(\d y)\right|^2\\
\nonumber&\leq 16K_2^2\sum_{i=1}^N\E|X^{i,N}_t-X^{i}_t|^2+4K_2^2\sum_{m=1}^N\E|X_t^{m,N}-X_t^{m}|^2+2c\E(1+|X_t^{1}|^2)\\
\nonumber&\leq C_0\sum_{i=1}^{N}\E|X_0^{i,N}-X_0^{i}|^2+C_0(1+\E|X_0^1|^{2}).
\end{align}
Substituting this into \eqref{ENW}, we get
 \beg{align}\label{fie13}\nonumber&\E \log F(X_{t}^{1,N},X_{t}^{2,N},\cdots, X_{t}^{N,N})\\
 &\le \log \E [ F(X_{t}^{1},X_{t}^{2},\cdots,X_{t}^{N})]+c(p)t\left(\sum_{i=1}^{N}\E|X_0^{i,N}-X_0^{i}|^2+(1+\E|X_0^1|^{2})\right)\\
\nonumber&+\frac{c(p)}{t}\sum_{i=1}^{N}\E|X_0^{i,N}-X_0^{i}|^2, \ \ 0< F\in \B_b((\R^d)^N), t\in(0,1].
\end{align}
Combining this with \cite[Theorem 1.4.2(2)]{Wbook},
\eqref{S13} and \eqref{uni00},
%and the fact
%$$\W_2(\L_{(X_t^{i,N})_{1\leq i\leq k}},\L_{(X_t^i)_{1\leq i\leq k}})^2\leq \frac{k}{N}\W_2(\L_{(X_t^{i,N})_{1\leq i\leq N}},\L_{(X_t^i)_{1\leq i\leq N}})^2,\ \ 1\leq k\leq N,$$
we may apply Theorem \ref{GRS}(i) for $r_0=0,t_0=1$ to derive the first assertion in (2).

Finally, combining \eqref{WDD} with \eqref{S-S} and \eqref{b-b}, we derive
\begin{align}\label{bxb}\nonumber&2\sum_{i=1}^N\<b(x_i,\frac{1}{N}\sum_{i=1}^N\delta_{x_i})-b(\tilde{x}_i,\frac{1}{N} \sum_{i=1}^N\delta_{\tilde{x}_i}),x_i-\tilde{x}_i\>\\
\nonumber&+\sum_{i=1}^N\|\sigma(x_i,\frac{1}{N}\sum_{i=1}^N\delta_{x_i})-\sigma(\tilde{x}_i,\frac{1}{N} \sum_{i=1}^N\delta_{\tilde{x}_i})\|_{HS}^2\\
&\leq -K_1\sum_{i=1}^N|x_i-\tilde{x}_i|^2+\sum_{i=1}^N 3K_2|x_i-\tilde{x}_i|^2+K_2\sum_{i=1}^N\frac{1}{N}\sum_{i=1}^N|x_i-\tilde{x}_i|^2\\
\nonumber&+K_2\sum_{i=1}^N|x_i-\tilde{x}_i|^2+K_2\sum_{i=1}^N\frac{1}{N}\sum_{i=1}^N|x_i-\tilde{x}_i|^2\\
\nonumber&\leq (-K_1+6K_2)|x-\tilde{x}|^2,\ \ x=(x_1,x_2,\cdots,x_N),\tilde{x}=(\tilde{x}_1,\tilde{x}_2,\cdots,\tilde{x}_N)\in(\R^d)^N.
\end{align}
Recall that $K_1>8K_2$. \eqref{bxb} implies that for any $N\geq 1$, $(X_t^{1,N},X_t^{2,N},\cdots,X_t^{N,N})$ has a unique invariant probability measure $\bar{\mu}^N$ which satisfies
\begin{align}\label{bau}&\nonumber\W_2(\L_{(X_t^{1,N}X_t^{1,N},\cdots,X_t^{N,N})},\bar{\mu}^N)^2\\
&\leq \e^{-(K_1-6K_2)t}\W_2(\L_{(X_0^{1,N}X_0^{1,N},\cdots,X_0^{N,N})},\bar{\mu}^N)^2,\ \ N\geq 1, t\geq 0.
\end{align}
So, the last two assertions in (2) follow from \eqref{INV}, \eqref{uni00}, \eqref{bau} and Theorem \ref{GRS}(ii)-(iii).
\end{proof}
\begin{rem}
\eqref{DDT} is nothing else but
 the estimate of relative entropy for two diffusion processes with different drifts and the same initial value. More precisely, consider
 $$\d Y_t^i=b^i_t(Y_t^i)\d t+\d W_t,\ \ i=1,2,$$
with $Y_0^1=Y_0^2$, then it holds
\begin{align*}
&\mathrm{Ent}(\L_{Y_{t}^{1}}|\L_{Y_{t}^{2}})\leq \frac{1}{2}\int_0^t\E|b^1_s(Y_s^1)-b^2_s(Y_s^1)|^2\d s,\ \ t\geq 0.
\end{align*}
The trick to derive \eqref{DDT} is coupling by change of measure. In fact,
\eqref{DDT} can also be gained by the entropy method in \cite{BJW,JW,JW1}.
\end{rem}
\section{Application in path dependent models with additive noise}
Throughout this section, fix a constant $r_0\geq0$. Let $\C^d= C([-r_0,0];\mathbb{R}^{d})$ be equipped with the uniform norm $\|\xi\|_\infty =:\sup_{s\in[-r,0]} |\xi(s)|$. For any $f\in C([-r_0,\infty);\mathbb{R}^{d})$, $t\geq 0$, define $f_t \in \C^d$ as $f_t(s)=f(t+s), s\in [-r_0,0]$. Recall that $\scr P(\C^d)$ is the set of all probability measures in $\C^d$ equipped with the weak topology and
$$\scr P_2(\C^d) = \big\{\mu\in \scr P(\C^d): \mu(\|\cdot\|_\infty^2)<\infty\big\}.$$
$\scr P_2(\C^d)$ is a Polish space under the $L^2$-Wasserstein distance
$$\W_2(\mu,\nu):= \inf_{\pi\in \mathbf{C}(\mu,\nu)} \bigg(\int_{\C^d\times\C^d} \|\xi-\eta\|_\infty^2 \pi(\d \xi,\d \eta)\bigg)^{\ff 1 {2}},\ \ \mu,\nu\in \scr P_{2}(\C^d),$$ where $\mathbf{C}(\mu,\nu)$ is the set of all couplings of $\mu$ and $\nu$.
As in Section 2, for any $k\geq 1$, define the Wasserstein distance on $\scr P_2((\C^d)^k)$:
$$\W_2(\mu,\nu):= \inf_{\pi\in \mathbf{C}(\mu,\nu)} \bigg(\int_{(\C^d)^k\times(\C^d)^k} \sum_{i=1}^k\|\xi^i-\tilde{\xi}^i\|_\infty^2\pi(\d \xi,\d \tilde{\xi})\bigg)^{\ff 1 {2}},\ \ \mu,\nu\in \scr P_{2}((\C^d)^k),$$
where $$\xi=(\xi^1,\xi^2,\cdots,\xi^k), \tilde{\xi}=(\tilde{\xi}^1,\tilde{\xi}^2,\cdots,\tilde{\xi}^k)\in(\C^d)^k.$$
%Note that
%$$\tilde{\W}_2(\mu,\nu)\geq \W_2(\mu,\nu),\ \ \mu,\nu\in \scr P_{2}((\C^d)^N).$$
\subsection{The case with non-degenerate diffusion coefficients}
Let $W(t)$ be an $n$-dimensional Brownian motion on some complete filtration probability space $(\Omega, \scr F, (\scr F_t)_{t\geq 0},\P)$. Let $b:\R^d\to\R^d$, $B:\C^d\times\scr P(\C^d)\to\R^d$ and $\sigma: \C^d\times\scr P(\C^d)\to\R^d\otimes\R^{n}$ be measurable and bounded on bounded set. Consider path dependent McKean-Vlasov SDEs:
\begin{align}\label{EPD}\d  X(t)=b(X(t))\mathrm{d} t+B(X_t,\L_{X_t})\d t+\sigma(X_t,\L_{X_t})\mathrm{d} W(t),
\end{align}
%\subsection{Propagation of Chaos for bounded interaction}

Let $X_0$ be an $\F_0$-measurable and $\C^d$-valued random variable,
$N\ge1$ be an integer and $(X_0^i,W^i(t))_{1\le i\le N}$ be i.i.d.\,copies of $(X_0,W(t)).$ Consider
\begin{align*}\d  X^i(t)=b(X^i(t))\mathrm{d} t+B(X^i_t,\L_{X^i_t})\d t+\sigma(X^i_t,\L_{X^i_t})\mathrm{d} W^i(t),\ \ 1\leq i\leq N,
\end{align*}
and the mean field interacting particle system:
\begin{align*}\d X^{i,N}(t)&=
b(X^{i,N}(t))\mathrm{d} t+B(X_t^{i,N}, \ff{1}{N}\sum_{j=1}^N\dd_{X_t^{j,N}})\d t\\
&+\sigma(X_t^{i,N}, \ff{1}{N}\sum_{j=1}^N\dd_{X_t^{j,N}})\mathrm{d} W^i(t),\ \ 1\leq i\leq N.
\end{align*}

We assume
\begin{enumerate}
\item[{\bf(C1)}] $\sigma(\xi,\mu)=\int_{\C^d}\tilde{\sigma}(\xi,\eta)\mu(\d \eta)$. There exists a constant $K_\sigma>0$ and
\begin{align*}
\|\tilde{\sigma}(\xi,\eta)-\tilde{\sigma}(\tilde{\xi},\tilde{\eta})\|^2_{HS}\leq K_\sigma(\|\xi-\tilde{\xi}\|^2_\infty+\|\eta-\tilde{\eta}\|_\infty^2),\ \ \xi,\tilde{\xi},\eta,\tilde{\eta}\in\C^d.
\end{align*}
\item[{\bf(C2)}] $b$ is continuous and there exists a constant $K_b\geq0$ such that
    $$2\<b(x)-b(y),x-y\>\leq -K_b|x-y|^2,\ \ x,y\in\R^d.$$
\item[{\bf(C3)}]  $B(\xi,\mu)=\int_{\C^d}\tilde{B}(\xi,\eta)\mu(\d \eta)$. There exists a constant $K_B>0$ such that
$$|\tilde{B}(\xi,\eta)-\tilde{B}(\tilde{\xi},\tilde{\eta})|\leq K_B(\|\xi-\tilde{\xi}\|_\infty+\|\eta-\tilde{\eta}\|_\infty),\ \ \xi,\tilde{\xi},\eta,\tilde{\eta}\in\C^d.$$
\end{enumerate}
\begin{thm}\label{PDP}
Assume {\bf(C1)}-{\bf(C3)} and $\L_{X_0^{1,N}},\L_{X_0^{1}}\in \scr P_2(\C^d)$. If
\begin{align}\label{vaS}72K_\sigma+8K_B<\sup_{v\in[0,K_b]}v\e^{-v r_0},
\end{align}
then the following assertions hold.

(1) There exists a constant $c>0$ such that
\begin{equation}\begin{split}\label{S13n}
&\E\sum_{i=1}^N\|X^i_s-X^{i,N}_s\|_\infty^2\\
&\leq 2\e^{K_b r_0} \e^{-\e^{K_b r_0}\Lambda s}\E\sum_{i=1}^N\|X_0^{i,N}-X_0^{i}\|_{\infty}^2+c(1+\E\|X_0^1\|_\infty^{2}),\ \ s\geq 0,
\end{split}\end{equation}
where \begin{align}\label{lam}
\Lambda=\frac{1}{2}\left\{\sup_{v\in[0,K_b]}v\e^{-v r_0}-(72K_\sigma+8K_B)\right\}.
\end{align}
(2) If $\sigma$ is a constant matrix and $\sigma\sigma^\ast$ is invertible, then there exists a constant $c>0$ such that
\begin{align*}&\mathrm{Ent}(\L_{(X_{t}^{1,N},X_{t}^{2,N},\cdots, X_{t}^{k,N})}|\L_{(X_{t}^{1},X_{t}^{2},\cdots,X_{t}^{k})})\\
&\leq c\frac{k}{N}\frac{\e^{-\e^{K_b r_0}\Lambda t}}{(t-r_0)\wedge 1}\W_2(\L_{(X_{0}^{i,N})_{1\leq i\leq N}},\L_{(X_{0}^{i})_{1\leq i\leq N}})^2+\frac{ck}{N}(1+\E\|X_0^1\|_\infty^{2}),\ \ 1\leq k\leq N,t>r_0.
\end{align*}
Consequently, it holds
\begin{align*}&\mathrm{Ent}(\L_{(X_{t}^{1,N},X_{t}^{2,N},\cdots, X_{t}^{k,N})}|(\mu^\ast)^{\otimes k})\\
&\leq c\frac{k}{N}\frac{\e^{-\e^{K_b r_0}\Lambda t}}{(t-r_0)\wedge 1}\W_2(\L_{(X_{0}^{i,N})_{1\leq i\leq N}},(\mu^\ast)^{\otimes N})^2+\frac{ck}{N}(1+\mu^\ast(\|\cdot\|_\infty^{2})),\ \ 1\leq k\leq N,t>r_0,
\end{align*}
and
\begin{align*}\mathrm{Ent}(\bar{\mu}^N\circ\pi_k^{-1}|(\mu^\ast)^{\otimes k})\leq ck\frac{1+\mu^\ast(\|\cdot\|_\infty^2)}{N},\ \ 1\leq k\leq N,
\end{align*}
where $\pi_k$ is the projecting mapping from $(\C^d)^N$ to $(\C^d)^k$, $\bar{\mu}^N$ and $\mu^\ast$ are the unique invariant probability measures of $(X_t^{1,N}, X_t^{2,N},\cdots,X_t^{N,N})$ and $X_t^1$ respectively.
\end{thm}
\begin{proof}
Firstly, it is standard to derive from {\bf(C1)}-{\bf(C3)} that for any $T>0$,
\begin{align}\label{mmf} \E\left[1+\sup_{t\in[-r_0,T]}|X^{1}(t)|^2\right]<c_{T}\left(1+\E\|X^1_0\|_\infty^2\right)
\end{align}
and
\begin{align}\label{grf} \E\left[1+\sum_{i=1}^N\sup_{t\in[-r,T]}|X^{i,N}(t)|^2\right]<\tilde{c}_{T,N}\left(1+\E\sum_{i=1}^N\|X^{i,N}_0\|_\infty^2\right)
\end{align}
hold for some constants $c_{T}>0$ and $\tilde{c}_{T,N}>0$.

(1) It follows from It\^{o}'s formula and {\bf(C2)} that
\begin{align}\label{CCS}
\nonumber&\d |X^{i,N}(t)-X^i(t)|^2\leq -K_b|X^{i,N}(t)-X^i(t)|^2\d t\\
&+2\left\<\frac{1}{N}\sum_{m=1}^N  \tilde{B}(X_t^{i,N},X_t^{m,N})-\int_{\C^d} \tilde{B}(X_t^{i},y)\L_{X_t^1}(\d y),X^{i,N}(t)-X^{i}(t)\right\>\d t\\
\nonumber&+\left\|\frac{1}{N}\sum_{m=1}^N  \tilde{\sigma}(X_t^{i},X_t^{m})-\int_{\C^d} \tilde{\sigma}(X_t^{i},y)\L_{X_t^1}(\d y)\right\|_{HS}^2\d t+\d M^i(t),
\end{align}
where \begin{align*}\d M^i(t)&=2\bigg\<\left[\frac{1}{N}\sum_{m=1}^N  \tilde{\sigma}(X_t^{i,N},X_t^{m,N})-\int_{\C^d} \tilde{\sigma}(X_t^{i},y)\L_{X_t^1}(\d y)\right]\d W^i(t),\\
&\qquad\qquad\quad X^{i,N}(t)-X^{i}(t)\bigg\>.
\end{align*}
Similar to \eqref{b1x} and \eqref{sigm0}, we derive from {\bf (C1)} and {\bf(C3)} that
\begin{align}\label{sigmt}
\nonumber&\left\|\frac{1}{N}\sum_{m=1}^N  \tilde{\sigma}(X_t^{i,N},X_t^{m,N})-\int_{\C^d} \tilde{\sigma}(X_t^{i},y)\L_{X_t^1}(\d y)\right\|_{HS}^2\\
&\leq 2K_\sigma\|X^{i,N}_t-X^{i}_t\|_{\infty}^2+2K_\sigma\frac{1}{N}\sum_{m=1}^N \|X_t^{m,N}-X_t^{m}\|_\infty ^2\\
\nonumber&+2\left\|\frac{1}{N}\sum_{m=1}^N  \tilde{\sigma}(X_t^{i},X_t^{m})-\int_{\C^d} \tilde{\sigma}(X_t^{i},y)\L_{X_t^1}(\d y)\right\|_{HS}^2,
\end{align}
and
\begin{align}\label{BMT}
\nonumber&2\left\<\frac{1}{N}\sum_{m=1}^N  \tilde{B}(X_t^{i,N},X_t^{m,N})-\int_{\C} \tilde{B}(X_t^{i},y)\L_{X_t^1}(\d y),X^{i,N}(t)-X^{i}(t)\right\>\\
&\leq 3K_B\|X^{i,N}_t-X^{i}_t\|_{\infty}^2+K_B\frac{1}{N}\sum_{m=1}^N \|X_t^{m,N}-X_t^{m}\|_\infty ^2\\
\nonumber&+\frac{2}{\Lambda}\left|\frac{1}{N}\sum_{m=1}^N\tilde{B}(X_t^{i}, X_t^{m})-\int_{\C^d}\tilde{B}(X_t^{i}, y)\L_{X_t^1}(\d y)\right|^2+\frac{\Lambda}{2}|X^{i,N}(t)-X^{i}(t)|^2,
\end{align}
here $\Lambda$ is defined in \eqref{lam}.
Substituting \eqref{sigmt} and \eqref{BMT} into \eqref{CCS}, we get
\begin{align}\label{Itf}\nonumber&\d [\e^{K_bt}|X^{i,N}(t)-X^i(t)|^2]\leq \e^{K_b t}(3K_B+2K_\sigma)\|X_t^{i,N}-X_t^{i}\|_\infty ^2\d t\\
\nonumber&+\e^{K_b t}(K_B+2K_\sigma)\frac{1}{N}\sum_{m=1}^N \|X_t^{m,N}-X_t^{m}\|_\infty ^2\d t+\e^{K_b t}\frac{\Lambda}{2}|X^{i,N}(t)-X^{i}(t)|^2\d t\\
&+\e^{K_b t}\frac{2}{\Lambda}\left|\frac{1}{N}\sum_{m=1}^N\tilde{B}(X_t^{i}, X_t^{m})-\int_{\C^d}\tilde{B}(X_t^{i}, y)\L_{X_t^1}(\d y)\right|^2\d t\\
\nonumber&+2\e^{K_b t}\left\|\frac{1}{N}\sum_{m=1}^N  \tilde{\sigma}(X_t^{i},X_t^{m})-\int_{\C^d} \tilde{\sigma}(X_t^{i},y)\L_{X_t^1}(\d y)\right\|_{HS}^2\d t+\e^{K_b t}\d M^i(t).
\end{align}
Let $\eta_t=\sum_{i=1}^N\sup_{s\in[-r_0,t]}\e^{K_bs^+}| X^{i,N}(s)-X^{i}(s)|^2$.  We conclude from \eqref{Itf} that
\begin{align}\label{ETa}\nonumber\E\eta_t&\leq \sum_{i=1}^N\E\|X^{i,N}_0-X^i_0\|^2_\infty+\int_0^t\e^{K_b s}\left(4K_B+4K_\sigma+\frac{\Lambda}{2}\right)\E\sum_{i=1}^N\|X_s^{i,N}-X_s^{i}\|_\infty ^2\d s\\
&+\int_0^t\e^{K_b s}\frac{2}{\Lambda}\sum_{i=1}^N\E\left|\frac{1}{N}\sum_{m=1}^N\tilde{B}(X_s^{i}, X_s^{m})-\int_{\C^d}\tilde{B}(X_s^{i}, y)\L_{X_s^1}(\d y)\right|^2\d s\\
\nonumber&+\int_0^t2\e^{K_b s}\sum_{i=1}^N\E\left\|\frac{1}{N}\sum_{m=1}^N  \tilde{\sigma}(X_s^{i},X_s^{m})-\int_{\C^d} \tilde{\sigma}(X_s^{i},y)\L_{X_s^1}(\d y)\right\|_{HS}^2\d s\\
\nonumber&+\sum_{i=1}^N\E\sup_{r\in[0,t]}\int_0^r\e^{K_b s}\d M^i(s).
\end{align}
Applying BDG's inequality, we derive from \eqref{sigmt} that
\begin{align}\label{mgy}
\nonumber&\sum_{i=1}^N\E\sup_{r\in[0,t]}\int_0^r\e^{K_b s}\d M^i(s)\\
\nonumber&\leq \sum_{i=1}^N4\E\bigg\{\int_0^t\e^{2K_b s}\left\|\frac{1}{N}\sum_{m=1}^N  \tilde{\sigma}(X_s^{i,N},X_s^{m,N})-\int_{\C^d} \tilde{\sigma}(X_s^{i},y)\L_{X_s^1}(\d y)\right\|^2\\
\nonumber&\qquad\qquad\quad\times| X^{i,N}(s)-X^{i}(s)|^2\d s\bigg\}^{\frac{1}{2}}\\
&\leq \frac{1}{2}\E\eta_t+8\sum_{i=1}^N\E\int_0^t\e^{K_b s}\left\|\frac{1}{N}\sum_{m=1}^N  \tilde{\sigma}(X_s^{i,N},X_s^{m,N})-\int_{\C^d} \tilde{\sigma}(X_s^{i},y)\L_{X_s^1}(\d y)\right\|^2\d s\\
\nonumber&\leq \frac{1}{2}\E\eta_t+16K_\sigma\sum_{i=1}^N\E\int_0^t\e^{K_b s}\|X^{i,N}_s-X^{i}_s\|_{\infty}^2\d s\\
\nonumber&+16K_\sigma\sum_{i=1}^N\E\int_0^t\e^{K_b s}\frac{1}{N}\sum_{m=1}^N \|X_s^{m,N}-X_s^{m}\|_\infty ^2\d s\\
\nonumber&+16\sum_{i=1}^N\E\int_0^t\e^{K_b s}\left\|\frac{1}{N}\sum_{m=1}^N  \tilde{\sigma}(X_s^{i},X_s^{m})-\int_{\C^d} \tilde{\sigma}(X_s^{i},y)\L_{X_s^1}(\d y)\right\|_{HS}^2\d s.
\nonumber\end{align}
Combining this with \eqref{ETa}, the fact $\eta_s\geq \sum_{i=1}^N\e^{K_b(s-r_0)}\|X_s^{i,N}-X_s^i\|_\infty^2$ and \eqref{mmf}-\eqref{grf}, we obtain
\begin{align*}
\E\eta_t&\leq 2\sum_{i=1}^N\E\|X^{i,N}_0-X^i_0\|_\infty^2+\e^{K_b r_0}[72K_\sigma+8K_B+\Lambda]\int_0^t\E\eta_s\d s\\
&+\int_0^t\e^{K_b s}\frac{4}{\Lambda}\sum_{i=1}^N\E\left|\frac{1}{N}\sum_{m=1}^N\tilde{B}(X_s^{i}, X_s^{m})-\int_{\C^d}\tilde{B}(X_s^{i}, y)\L_{X_s^1}(\d y)\right|^2\d s\\
&+\int_0^t36\e^{K_b s}\sum_{i=1}^N\E\left\|\frac{1}{N}\sum_{m=1}^N  \tilde{\sigma}(X_s^{i},X_s^{m})-\int_{\C^d} \tilde{\sigma}(X_s^{i},y)\L_{X_s^1}(\d y)\right\|_{HS}^2\d s.
\end{align*}
Gronwall's inequality implies that
\begin{align*}\E\eta_t&\leq 2\e^{\e^{K_b r_0}[72K_\sigma+8K_B+\Lambda]t}\sum_{i=1}^N\E\|X^{i,N}_0-X^i_0\|_\infty^2\\
&+\int_0^t\e^{\e^{K_b r_0}[72K_\sigma+8K_B+\Lambda](t-s)}\e^{K_b s}\frac{4}{\Lambda}\\
&\qquad\qquad\quad\times\sum_{i=1}^N\E\left|\frac{1}{N}\sum_{m=1}^N\tilde{B}(X_s^{i}, X_s^{m})-\int_{\C^d}\tilde{B}(X_s^{i}, y)\L_{X_s^1}(\d y)\right|^2\d s\\
&+\int_0^t36\e^{\e^{K_b r_0}[72K_\sigma+8K_B+\Lambda](t-s)}\e^{K_b s}\\
&\qquad\qquad\quad\times\sum_{i=1}^N\E\left\|\frac{1}{N}\sum_{m=1}^N  \tilde{\sigma}(X_s^{i},X_s^{m})-\int_{\C^d} \tilde{\sigma}(X_s^{i},y)\L_{X_s^1}(\d y)\right\|_{HS}^2\d s.
\end{align*}
Again noting $\eta_t\geq \sum_{i=1}^N\e^{K_b(t-r_0)}\|X_t^{i,N}-X_t^i\|_\infty^2$, we arrive at
\begin{align}\label{MIT}
\nonumber&\sum_{i=1}^N\E \|X_t^{i,N}-X_t^i\|_\infty^2\\
\nonumber&\leq 2\e^{K_b r_0}\e^{\e^{K_b r_0}\{72K_\sigma+8K_B-K_b\e^{-K_b r_0}+\Lambda \}t}\sum_{i=1}^N\E\|X^{i,N}_0-X^i_0\|_\infty^2\\
&+\e^{K_b r_0}\frac{4}{\Lambda}\int_0^t\e^{\e^{K_b r_0}[72K_\sigma+8K_B-K_b\e^{-K_b r_0}+\Lambda](t-s)}\\
\nonumber&\qquad\qquad\quad\times \sum_{i=1}^N\E\left|\frac{1}{N}\sum_{m=1}^N\tilde{B}(X_s^{i}, X_s^{m})-\int_{\C^d}\tilde{B}(X_s^{i}, y)\L_{X_s^1}(\d y)\right|^2\d s\\
\nonumber&+36\e^{K_b r_0}\int_0^t\e^{\e^{K_b r_0}[72K_\sigma+8K_B-K_b\e^{-K_b r_0}+\Lambda](t-s)}\\
\nonumber&\qquad\qquad\quad\times\sum_{i=1}^N\E\left\|\frac{1}{N}\sum_{m=1}^N  \tilde{\sigma}(X_s^{i},X_s^{m})-\int_{\C^d} \tilde{\sigma}(X_s^{i},y)\L_{X_s^1}(\d y)\right\|_{HS}^2\d s.
\end{align}
Since {\bf(C2)} holds for any $v\in[0,K_b]$, replacing $K_b$ with any $v\in[0,K_b]$ in \eqref{MIT} and combining with \eqref{lam}, we conclude that
\begin{align}\label{RES}
\nonumber&\sum_{i=1}^N\E \|X_t^{i,N}-X_t^i\|_\infty^2\\
\nonumber&\leq 2\e^{K_b r_0}\e^{-\e^{K_b r_0}\Lambda t}\sum_{i=1}^N\E\|X^{i,N}_0-X^i_0\|_\infty^2\\
&+\e^{K_b r_0}\frac{4}{\Lambda}\int_0^t\e^{-\e^{K_b r_0}\Lambda(t-s)}\sum_{i=1}^N\E\left|\frac{1}{N}\sum_{m=1}^N\tilde{B}(X_s^{i}, X_s^{m})-\int_{\C^d}\tilde{B}(X_s^{i}, y)\L_{X_s^1}(\d y)\right|^2\d s\\
\nonumber&+36\e^{K_b r_0}\int_0^t\e^{-\e^{K_b r_0}\Lambda(t-s)}\sum_{i=1}^N\E\left\|\frac{1}{N}\sum_{m=1}^N  \tilde{\sigma}(X_s^{i},X_s^{m})-\int_{\C^d} \tilde{\sigma}(X_s^{i},y)\L_{X_s^1}(\d y)\right\|_{HS}^2\d s.
\end{align}

Next, similar to \eqref{S-S} and \eqref{b-b}, {\bf(C1)} and {\bf(C3)} imply that
\begin{align}\label{SDS}&\|\sigma(\xi,\mu)-\sigma(\tilde{\xi},\nu)\|_{HS}^2
\leq K_\sigma(\|\xi-\tilde{\xi}\|_\infty^2+\W_2(\mu,\nu)^2),\ \ \xi,\tilde{\xi}\in\C^d,\mu,\nu\in\scr P_2(\C^d)
\end{align}
and
\begin{align}\label{BLI}
\nonumber&2\<B(\xi,\mu)-B(\tilde{\xi},\nu),\xi(0)-\tilde{\xi}(0)\>\\
&\leq 3K_B\|\xi-\tilde{\xi}\|_\infty^2+K_B\W_1(\mu,\nu)^2,\ \ \xi,\tilde{\xi}\in\C^d,\mu,\nu\in\scr P_2(\C^d).
\end{align}
\eqref{SDS}-\eqref{BLI} together with {\bf(C2)} yield that for any $\xi,\tilde{\xi}\in\C,\mu,\nu\in\scr P_2(\C^d)$,
\begin{align}\label{TTS}
\nonumber&2\<b(\xi(0))-b(\tilde{\xi}(0))+B(\xi,\mu)-B(\tilde{\xi},\nu),\xi(0)-\tilde{\xi}(0)\>+\|\sigma(\xi,\mu)-\sigma(\tilde{\xi},\nu)\|_{HS}^2\\
&\leq -K_b|\xi(0)-\tilde{\xi}(0)|^2+(3K_B+K_\sigma)\|\xi-\tilde{\xi}\|_\infty^2+(K_B+K_\sigma)\W_2(\mu,\nu)^2,
\end{align}
\begin{align}\label{s-s}
\|\sigma(\xi,\mu)\|_{HS}^2\leq 2K_\sigma(\|\xi\|_\infty^2+\mu(\|\cdot\|_\infty^2))+2\|\sigma(0,\delta_\mathbf{0})\|_{HS}^2,
\end{align}
and
\begin{align}\label{MOM}
\nonumber&2\<b(\xi(0))+B(\xi,\mu),\xi(0)\>+\|\sigma(\xi,\mu)\|_{HS}^2\\
\nonumber&\leq -K_b|\xi(0)|^2+(3K_B+2K_\sigma)\|\xi\|_\infty^2+(K_B+2K_\sigma)\mu(\|\cdot\|_\infty^2)\\
&\quad+2\<b(0)+B(0,\delta_\mathbf{0}),\xi(0)\>+2\|\sigma(0,\delta_\mathbf{0})\|_{HS}^2\\
\nonumber&\leq -K_b|\xi(0)|^2+(3K_B+2K_\sigma)\|\xi\|_\infty^2+(K_B+2K_\sigma)\mu(\|\cdot\|_\infty^2)\\
\nonumber&\quad+\frac{\Lambda}{2}|\xi(0)|^2+\frac{2}{\Lambda}|b(0)+B(0,\delta_\mathbf{0})|^2+2\|\sigma(0,\delta_\mathbf{0})\|_{HS}^2,
\end{align}
here $\mathbf{0}$ is the zero element in $\C^d$. By It\^{o}'s formula and \eqref{MOM}, we obtain
\begin{align}\label{UNI}
\nonumber&\d [\e^{K_bt}|X^1(t)|^2]\leq \e^{K_b t}\left(3K_B+2K_\sigma+\frac{\Lambda}{2}\right)\|X_t^{1}\|_\infty ^2\d t+\e^{K_b t}(K_B+2K_\sigma)\E\|X_t^{1}\|_\infty ^2\d t\\
&+\e^{K_b t}\left\{\frac{2}{\Lambda}|b(0)+B(0,\delta_\mathbf{0})|^2+2\|\sigma(0,\delta_\mathbf{0})\|_{HS}^2\right\}\d t\\
\nonumber&+2\e^{K_b t}\<\sigma(X_t^1,\L_{X_t^1})\d W^1(t),X^1(t)\>.
\end{align}
Observing \eqref{vaS} and applying the same argument to derive \eqref{RES}, we conclude from \eqref{s-s} and \eqref{UNI} that there exists a constant $C>0$ such that
\begin{align}\label{uni}
\sup_{t\geq 0}\E(1+\|X_t^1\|_\infty^2)<C(1+\E\|X_0^1\|_\infty^2).
\end{align}
Finally, combining \eqref{RES} with Lemma \ref{CTY} and \eqref{uni}, we derive \eqref{S13n}.

Furthermore, by \eqref{vaS}, \eqref{SDS}, \eqref{TTS}, \eqref{uni} and \cite[Remark 2.1]{HRW},
\eqref{EPD} has a unique invariant probability measure $\mu^\ast$ with
\begin{align}\label{INP}
\W_2(\L_{X_t^1},\mu^\ast)^2\leq C\e^{-\lambda t}\W_2(\L_{X_0^1},\mu^\ast)^2,\ \ t\geq 0
\end{align}
for some constant $C,\lambda>0$.

(2) We first prove that for any $N\geq 1$, $(X_t^{1,N},X_t^{2,N},\cdots,X_t^{N,N})$ has a unique invariant probability measure $\bar{\mu}^N$. Let $\xi=(\xi_1,\xi_2,\cdots,\xi_N),\tilde{\xi}=(\tilde{\xi}_1,\tilde{\xi}_2,\cdots,\tilde{\xi}_N)\in(\C^d)^N$. Observe that \eqref{SDS}-\eqref{TTS} yield
\begin{align}\label{XTY}
\nonumber&\left\|\sigma(\xi_i,\frac{1}{N}\sum_{i=1}^N\delta_{\xi_i})-\sigma(\tilde{\xi}_i,\frac{1}{N} \sum_{i=1}^N\delta_{\tilde{\xi}_i})\right\|_{HS}^2\\
\nonumber&=\left\|\frac{1}{N}\sum_{j=1}^N\{\tilde{\sigma}(\xi_i,\xi_j)-\tilde{\sigma} (\tilde{\xi}_i,\tilde{\xi}_j)\}\right\|_{HS}^2\\
&\leq \frac{1}{N}\sum_{j=1}^N\|\tilde{\sigma}(\xi_i,\xi_j)-\tilde{\sigma} (\tilde{\xi}_i,\tilde{\xi}_j)\|_{HS}^2\\
\nonumber&\leq \frac{1}{N}\sum_{j=1}^NK_\sigma(\|\xi_i-\tilde{\xi}_i\|^2_\infty+\|\xi_j-\tilde{\xi}_j\|_\infty^2)\\
\nonumber&\leq K_\sigma \|\xi_i-\tilde{\xi}_i\|^2_\infty+ \frac{1}{N}\sum_{j=1}^NK_\sigma\|\xi_j-\tilde{\xi}_j\|_\infty^2,
\end{align}
and
\begin{align}\label{y-x}
\nonumber&2\<b(\xi_i(0))-b(\tilde{\xi}_i(0)),\xi_i(0)-\tilde{\xi}_i(0)\>\\
&\nonumber+2\left\<B(\xi_i,\frac{1}{N}\sum_{i=1}^N\delta_{\xi_i}) -B(\tilde{\xi}_i,\frac{1}{N} \sum_{i=1}^N\delta_{\tilde{\xi}_i}),\xi_i(0)-\tilde{\xi}_i(0)\right\>\\
&+\left\|\sigma(\xi_i,\frac{1}{N}\sum_{i=1}^N\delta_{\xi_i})-\sigma(\tilde{\xi}_i,\frac{1}{N} \sum_{i=1}^N\delta_{\tilde{\xi}_i})\right\|_{HS}^2\\
\nonumber&\leq -K_b|\xi_i(0)-\tilde{\xi}_i(0)|^2+3K_B \|\xi_i-\tilde{\xi}_i\|^2_\infty +K_B\frac{1}{N}\sum_{j=1}^N\|\xi_j-\tilde{\xi}_j\|_\infty^2\\
\nonumber&+K_\sigma \|\xi_i-\tilde{\xi}_i\|^2_\infty+\frac{1}{N}\sum_{j=1}^NK_\sigma\|\xi_j-\tilde{\xi}_j\|_\infty^2.
\end{align}
Let $(X^{i,N}(t))_{1\leq i\leq N}$ and $(\tilde{X}^{i,N}(t))_{1\leq i\leq N}$ be the solutions to the mean field interacting particle system with initial values $(X^{i,N}_0)_{1\leq i\leq N}$ and $(\tilde{X}^{i,N}_0)_{1\leq i\leq N}$ respectively. Applying It\^{o}'s formula, it follows from \eqref{XTY} and \eqref{y-x} that
\begin{align}\label{Itg}\nonumber&\d [\e^{K_bt}|X^{i,N}(t)-\tilde{X}^{i,N}(t)|^2]\leq \e^{K_b t}(3K_B+K_\sigma)\|X_t^{i,N}-\tilde{X}^{i,N}(t)\|_\infty ^2\d t\\
&+\e^{K_b t}(K_B+K_\sigma)\frac{1}{N}\sum_{m=1}^N \|X_t^{m,N}-\tilde{X}_t^{m,N}\|_\infty ^2\d t+\e^{K_b t}\d \tilde{M}^i(t),
\end{align}
where
\begin{align*}\d \tilde{M}^i(t)&=2\bigg\<\left[\frac{1}{N}\sum_{m=1}^N  \tilde{\sigma}(X_t^{i,N},X_t^{m,N})-\frac{1}{N}\sum_{m=1}^N  \tilde{\sigma}(\tilde{X}_t^{i,N},\tilde{X}_t^{m,N})\right]\d W^i(t),\\
&\qquad\qquad\quad X^{i,N}(t)-\tilde{X}^{i,N}(t)\bigg\>.
\end{align*}
Similar to \eqref{mgy}, we conclude from BDG's inequality and \eqref{XTY} that
\begin{align*}
&\sum_{i=1}^N\E\sup_{r\in[0,t]}\int_0^r\e^{K_b s}\d \tilde{M}^i(s)\leq \frac{1}{2}\E\tilde{\eta}_t+16K_\sigma\sum_{i=1}^N\E\int_0^t\e^{K_b s}\|X^{i,N}_s-\tilde{X}^{i,N}_s\|_{\infty}^2\d s
\end{align*}
for $\tilde{\eta}_t=\sum_{i=1}^N\sup_{s\in[-r_0,t]}\e^{K_bs^+}| X^{i,N}(s)-\tilde{X}^{i,N}(s)|^2$.

By the same argument to derive \eqref{MIT}, it is not difficult to see from \eqref{Itg} that
\begin{align}\label{Mkt}\sum_{i=1}^N\E \|X_t^{i,N}-\tilde{X}_t^{i,N}\|_\infty^2&\leq 2\e^{K_b r_0}\e^{\e^{K_b r_0}\{36K_\sigma+8K_B-K_b\e^{-K_b r_0} \}t}\sum_{i=1}^N\E\|X^{i,N}_0-\tilde{X}_0^{i,N}\|_\infty^2.
\end{align}
Let
\begin{align}\label{lamti}
\tilde{\Lambda}=\sup_{v\in[0,K_b]}v\e^{-v r_0}-(36K_\sigma+8K_B).
\end{align}
Then $\tilde{\Lambda}>0$ due to \eqref{vaS}.
Again noting that {\bf(C2)} holds for any $v\in[0,K_b]$, replacing $K_b$ with any $v\in[0,K_b]$ in \eqref{Mkt} and combining with \eqref{lamti}, we have
\begin{align}\label{contr}\sum_{i=1}^N\E \|X_t^{i,N}-\tilde{X}_t^{i,N}\|_\infty^2&\leq 2\e^{K_b r_0}\e^{-\e^{K_b r_0}\tilde{\Lambda} t}\sum_{i=1}^N\E\|X^{i,N}_0-\tilde{X}_0^{i,N}\|_\infty^2.
\end{align}
Recall that $\xi=(\xi_1,\xi_2,\cdots,\xi_N)\in(\C^d)^N$. By \eqref{s-s} and \eqref{MOM}, we obtain
\begin{align}\label{s-s1}
\left\|\sigma(\xi_i,\frac{1}{N}\sum_{i=1}^N\delta_{\xi_i})\right\|_{HS}^2\leq 2K_\sigma(\|\xi_i\|_\infty^2+\frac{1}{N}\sum_{i=1}^N\|\xi_i\|_\infty^2)+2\|\sigma(\mathbf{0},\delta_\mathbf{0})\|_{HS}^2,
\end{align}
and
\begin{align}\label{MOM1}
\nonumber&2\<b(\xi_i(0)),\xi_i(0)\>+2\left\<B(\xi_i,\frac{1}{N}\sum_{i=1}^N\delta_{\xi_i}),\xi_i(0)\right\> +\left\|\sigma(\xi_i,\frac{1}{N}\sum_{i=1}^N\delta_{\xi_i})\right\|_{HS}^2\\
\nonumber&\leq -K_b|\xi_i(0)|^2+(3K_B+2K_\sigma)\|\xi_i\|_\infty^2+(K_B+2K_\sigma)\frac{1}{N}\sum_{i=1}^N\|\xi_i\|_\infty^2\\
&\quad+2\<b(0)+B(\mathbf{0},\delta_\mathbf{0}),\xi(0)\>+2\|\sigma(\mathbf{0},\delta_\mathbf{0})\|_{HS}^2\\
\nonumber&\leq -K_b|\xi_i(0)|^2+(3K_B+2K_\sigma)\|\xi_i\|_\infty^2+(K_B+2K_\sigma)\frac{1}{N}\sum_{i=1}^N\|\xi_i\|_\infty^2\\
\nonumber&\quad+\frac{\Lambda}{2}|\xi(0)|^2+\frac{2}{\Lambda}|b(0)+B(\mathbf{0},\delta_\mathbf{0})|^2+2\|\sigma(\mathbf{0},\delta_\mathbf{0})\|_{HS}^2,
\end{align}
here $\mathbf{0}$ is the zero element in $\C^d$. Again by It\^{o}'s formula and \eqref{MOM1}, we obtain
\begin{align}\label{UNI1}
\nonumber&\d [\e^{K_bt}|X^{i,N}(t)|^2]\leq \e^{K_b t}(3K_B+2K_\sigma)\|X_t^{i,N}\|_\infty ^2\d t\\
&+\e^{K_b t}(K_B+2K_\sigma)\frac{1}{N}\sum_{i=1}^N\|X_t^{i,N}\|_\infty ^2\d t\\
\nonumber&+\e^{K_b t}\frac{\Lambda}{2}|X^{i,N}(t)|^2\d t+\e^{K_b t}\left\{\frac{2}{\Lambda}|b(0)+B(\mathbf{0},\delta_\mathbf{0})|^2+2\|\sigma(\mathbf{0},\delta_\mathbf{0})\|_{HS}^2\right\}\d t\\
\nonumber&+2\e^{K_b t}\left\<\sigma(X_t^{i,N},\frac{1}{N}\sum_{i=1}^N\delta_{X_t^{i,N}})\d W^i(t),X^{i,N}(t)\right\>.
\end{align}
Observing \eqref{vaS} and applying the same argument to derive \eqref{uni}, we conclude from \eqref{s-s1} and \eqref{UNI1} that there exists a constant $C>0$ such that
\begin{align}\label{uni1}
\sup_{t\geq 0}\E\left(1+\sum_{i=1}^N\|X_t^{i,N}\|_\infty^2\right)<C\left(1+\sum_{i=1}^N\E\|X_0^{i,N}\|_\infty^2\right).
\end{align}
With \eqref{contr} and \eqref{uni1} in hand, we conclude from \cite[Remark 2.1]{HRW} that for any $N\geq 1$, $(X_t^{1,N},X_t^{2,N},\cdots,X_t^{N,N})$ has a unique invariant probability measure $\bar{\mu}^N$ which satisfies
\begin{align}\label{expco}\W_2(\L_{(X_t^{1,N}X_t^{1,N},\cdots,X_t^{N,N})},\bar{\mu}^N)^2\leq \tilde{C}\e^{-\tilde{\lambda } t}\W_2(\L_{(X_0^{1,N}X_0^{1,N},\cdots,X_0^{N,N})},\bar{\mu}^N)^2,\ \ t\geq 0
\end{align}
for some constants $\tilde{C},\tilde{\lambda}>0$.

Next, as in the proof of Theorem \ref{POC30}, we simply denote $\mu_t=\L_{X_t^i}, 1\leq i\leq N$ and let $\P^0, \E^0$ be defined in Theorem \ref{POC30}. Firstly, consider
\begin{align*}\d \bar{X}^{i}(t)&= b(\bar{X}^{i}(t))\d t+\int_{\C^d}\tilde{B}(\bar{X}_t^{i}, y)\mu_t(\d y)\d t+  \sigma\d W^i(t),\ \ \bar{X}_0^{i}=X_0^{i,N}, 1\leq i\leq N.
\end{align*}
By \cite[Theorem 4.2.1]{Wbook} and the fact that $\{(X^i,\bar{X}^{i})\}_{1\leq i\leq N}$ are independent under $\P^0$, we conclude that for any $p>1$, $t>r_0$, there exists a constant $c(p,t)$ independent of $N$ such that
 \begin{align*}\nonumber\left(\E^{0} [F(\bar{X}_t^{1},\bar{X}_t^{2},\cdots,\bar{X}_t^{N})]\right)^p
&\leq \E^{0} [F(X_t^{1},X_t^{2},\cdots,X_t^{N})^p]\\
&\times\exp\left\{c(p,t)\sum_{i=1}^N\|X_0^{i,N}-X_0^{i}\|_\infty^2\right\}, \ \ F\in \scr B^+_b((\C^d)^N).
\end{align*}
Similar to \eqref{A4}, we can find a constant $c>0$ such that
\begin{align}\label{A5}
\nonumber&\sum_{i=1}^N\E\left|\frac{1}{N}\sum_{m=1}^N  \tilde{B}(X_t^{i,N},X_t^{m,N})-\int_{\C^d} \tilde{B}(X_t^{i,N},y)\mu_t(\d y)\right|^2\\
&\leq 20K_B^2\sum_{i=1}^N\E\|X^{i,N}_t-X^{i}_t\|_\infty^2+2c\E(1+\|X_t^{1}\|_\infty^2).
%&\leq 20K_b^2\E|X^{1,N}_t-X^{1}_t|^2+\frac{2c}{N}c_{1,T}(1+\E|X_0^{1}|^2).
\end{align}
 Repeating the proof of \eqref{fie13}, we derive from \eqref{ENW}, \eqref{A5}, \eqref{S13n} and \eqref{uni} that for any $t>r_0$ and $p>1$, there exists a constant $c(p,t)>0$ such that
\begin{align}\label{enl}
\nonumber&\E \log F(X_{t}^{1,N},X_{t}^{2,N},\cdots, X_{t}^{N,N})\\
&\leq  \log \E [ F(X_{t}^{1},X_{t}^{2},\cdots,X_{t}^{N})]+c_0p\tilde{\W}_2(\L_{(X_{0}^{i,N})_{1\leq i\leq N}},\L_{(X_{0}^{i})_{1\leq i\leq N}})^2\\
\nonumber&+c_0p(1+\E\|X_0^1\|_\infty^2)t+ c(p,t)\tilde{\W}_2(\L_{(X_{0}^{i,N})_{1\leq i\leq N}},\L_{(X_{0}^{i})_{1\leq i\leq N}})^2, \ \ 0< F\in \B_b((\R^d)^N)
\end{align}
for some constant $c_0>0$.
Applying \eqref{S13n}, \eqref{uni}, \eqref{INP}, \eqref{enl}, \eqref{expco}, Theorem \ref{GRS} for $t_0=r_0+1$, we complete the proof.

\end{proof}
\subsection{The case with degenerate diffusion coefficients}
Let $m,d\in\mathbb{N}^{+}$. For any $x\in\R^{m+d}$, let $x^{(1)}\in\R^m$ and $x^{(2)}\in\R^d$ denote the first $m$ components and the last $d$ components of $x$ respectively.
In this section, we consider path-distribution dependent stochastic Hamiltonian system for $X(t)=(X^{(1)}(t),X^{(2)}(t))$:
\beq\label{EH}
\begin{cases}
\d X^{(1)}(t)=\{AX^{(1)}(t)+MX^{(2)}(t)\}\d t, \\
\d X^{(2)}(t)=\{b(X^{(2)}(t))+B(X_t,\L_{X_t})\}\d t+\sigma\d W(t),
\end{cases}
\end{equation}
where $W=(W(t))_{t\geq 0}$ is a $d$-dimensional standard Brownian motion with respect to a complete filtration probability space $(\OO, \F, \{\F_{t}\}_{t\ge 0}, \P)$, $A$ is an $m\times m$ matrix, $M$ is an $m\times d$ matrix, $\sigma$ is a $d\times d$ matrix, $b:\mathbb{R}^{d}\to \mathbb{R}^d$ and $B:\C^{m+d}\times \scr P(\C^{m+d})\to\mathbb{R}^d$ are measurable. Let $\{W^i_t\}_{i\geq 1}$ be independent $d$-dimensional Brownian motions. For any $1\leq i\leq N<\infty$, let $X^i$ solves \eqref{EH} with $W_t^i$ replacing $W_t$ and i.i.d. $\F_0$-measurable and $\C^{m+d}$-valued initial values $(X_0^i)_{1\leq i\leq N}$ and let $\{X^{i,N}\}_{1\leq i\leq N<\infty}$ be the mean field interacting particle system associated to \eqref{EH} with exchangeable initial value $\{X_0^{i,N}\}_{1\leq i\leq N<\infty}$.

To obtain the long time propagation of chaos, we make the following assumptions:
\begin{enumerate}
\item[\bf{(A1)}] There exist $K_1,K_2\geq 0$ such that
\begin{equation*}
2\<b(y)-b(\tilde{y}),y-\tilde{y}\>\leq -K_1|y-\tilde{y}|^2,\ \ |b(y)-b(\tilde{y})|\leq K_2|y-\tilde{y}|,\ \ y,\tilde{y}\in\R^d.
\end{equation*}
\item[\bf{(A2)}] There exists $K_A\geq 0$ such that
\begin{equation*}
2\<Ax-A\tilde{x},x-\tilde{x}\>\leq -K_A|x-\tilde{x}|^2,\ \ x,\tilde{x}\in\R^m.
\end{equation*}
\item[\bf{(A3)}] $B(\xi,\mu)=\int_{\C^{m+d}}\tilde{B}(\xi,\eta)\mu(\d \eta)$. There exists a constant $K_B>0$ such that
$$|\tilde{B}(\xi,\eta)-\tilde{B}(\tilde{\xi},\tilde{\eta})|\leq K_B(\|\xi-\tilde{\xi}\|_\infty+\|\eta-\tilde{\eta}\|_\infty),\ \ \xi,\tilde{\xi},\eta,\tilde{\eta}\in\C^{m+d}.$$
\item[\bf{(A4)}] $\sigma\sigma^\ast$ is invertible and there exists an integer $l$ with $0\leq l\leq m-1$ such that
$$\mathrm{Rank}[M,AM,\cdots,A^lM]=m.$$
\end{enumerate}
\begin{thm}\label{DEG}
Assume {\bf(A1)}-{\bf(A4)} and $\L_{X_0^{1,N}},\L_{X_0^{1}}\in \scr P_2(\C^{m+d})$. Let
$$\Lambda=\frac{1}{2}\left\{\sup_{v\in[0,K_1\wedge K_A]}v\e^{-v r_0}-(4K_B+2\|M\|)\right\}.$$
Assume $\Lambda>0$, i.e.
\begin{align}\label{vaS11}4K_B+2\|M\|<\sup_{v\in[0,K_1\wedge K_A]}v\e^{-v r_0}.
\end{align}
Then the assertions in Theorem \ref{PDP} hold.
\end{thm}
\begin{proof} By \eqref{BLI}, {\bf(A1)}-{\bf(A3)}, for any $\xi=(\xi^{(1)},\xi^{(2)}), \bar{\xi}=(\bar{\xi}^{(1)},\bar{\xi}^{(2)})\in \C^{m+d}$ and $\gamma,\bar{\gamma}\in \scr P_{2}(\C^{m+d})$,
\begin{align}\label{DIS}\nonumber&2\<A(\xi^{(1)}(0)-\bar{\xi}^{(1)}(0))+M(\xi^{(2)}(0)-\bar{\xi}^{(2)}(0)),\ \ \xi^{(1)}(0)-\bar{\xi}^{(1)}(0)\> \\
&+2\<b(\xi^{(2)}(0))-b(\bar{\xi}^{(2)}(0))+B(\xi,\gamma)-B(\bar{\xi}, \bar{\gamma}),\ \ \xi^{(2)}(0)-\bar{\xi}^{(2)}(0)\>\\
\nonumber&\leq -(K_1\wedge K_A)|\xi(0)-\bar{\xi}(0)|^2+(3K_B+2\|M\|)\|\xi-\bar{\xi}\|_\infty^2+K_B\W_2(\gamma,\bar{\gamma})^2,
\end{align}
which yields
\begin{align}\label{MON}
\nonumber&2\<A\xi^{(1)}(0)+M\xi^{(2)}(0),\ \ \xi^{(1)}(0)\>+2\<b(\xi^{(2)}(0))+B(\xi,\gamma), \xi^{(2)}(0)\>+\|\sigma\|_{HS}^2\\
&\leq -(K_1\wedge K_A)|\xi(0)|^2+(3K_B+2\|M\|)\|\xi\|_\infty^2+K_B\gamma(\|\cdot\|_\infty^2)\\
\nonumber&+\frac{2}{\Lambda}|b(0)+B(\mathbf{0},\delta_\mathbf{0})|^2+\frac{\Lambda}{2}|\xi(0)|^2+\|\sigma\|_{HS}^2.
\end{align}
Again simply denote $\mu_t=\L_{X_t^i}, 1\leq i\leq N$. Let $\{\bar{X}^i\}_{1\leq i\leq N}$ solve
\begin{equation*}
\begin{cases}
\d X^{(1)}(t)=\{AX^{(1)}(t)+MX^{(2)}(t)\}\d t, \\
\d X^{(2)}(t)=\{b(X^{(2)}(t))+B(X_t,\mu_t)\}\d t+\sigma\d W^i(t),\ \ 1\leq i\leq N
\end{cases}
\end{equation*}
with initial value $\{X_0^{i,N}\}_{1\leq i\leq N}$. Thanks to \cite[Lemma 4.1]{BWY} and the fact that $\{(X^i,\bar{X}^{i})\}_{1\leq i\leq N}$ are independent under $\P^0$, under {\bf(A1)}-{\bf(A4)}, for any $t>r_0$, there exists a constant $c(t)>0$ independent of $N$ such that for any $F\in \scr B^+_b((\C^{m+d})^N)$,
\begin{align*}\left(\E^{0} [F(\bar{X}_t^{1},\bar{X}_t^{2},\cdots,\bar{X}_t^{N})]\right)^2
&\leq \E^{0} [F(X_t^{1},X_t^{2},\cdots,X_t^{N})^2]\exp\left\{c(t)\sum_{i=1}^N\|X_0^{i,N}-X_0^{i}\|_\infty^2\right\}.
\end{align*}
Combining this with \eqref{vaS11}-\eqref{MON}, we can
repeat the proof of Theorem \ref{PDP} to derive the assertions in Theorem \ref{PDP}
and we complete the proof.
\end{proof}


\begin{thebibliography}{999}
%\bibitem{AZ0} G. B. Arous, O. Zeitouni, Increasing propagation of chaos for mean field models, \emph{Ann. Inst. Henri Poincar\'e Probab. Stat.} 35(1999), 85-102.
%\bibitem{B} D. Ban\~{o}s, The Bismut每Elworthy每Li formula for mean-field stochastic differential equations, \emph{Ann. Inst. Henri Poincar\'e Probab. Stat.} 54(2018),  220每233.
\bibitem{BWY} J. Bao, F.-Y. Wang, C. Yuan, Hypercontractivity for functional stochastic partial differential equations, \emph{Electron. J. Probab.} 20(2015), 1-15.

%\bibitem{RWBR18} V. Barbu, M. R\"ockner, Probabilistic representation for solutions to nonlinear {F}okker-{P}lanck equations, \emph{SIAM J. Math. Anal.} 50(2018), 4246--4260.

% \bibitem{RWBR} V. Barbu, M. R\"ockner, From nonlinear Fokker-Planck equations to solutions of distribution dependent SDE,  \emph{Ann. Probab.} 48(2020), 1902每1920.
% \bibitem{BRR}  V. Barbu,  M. R\"{o}ckner, F. Russo, Doubly probabilistic representation for the stochastic porous
%media type equation, \emph{Ann. Inst. Henri Poincar\'e Probab. Stat.} 53(2017), 2043每2073.
%\bibitem{BB} M. Bauer, T. M-Brandis, Existence and Regularity of Solutions to Multi-Dimensional Mean-Field Stochastic Differential Equations with Irregular Drift,	\emph{arXiv:1912.05932}.
%\bibitem{BBP} M. Bauer, T. M-Brandis, F. Proske, Strong Solutions of Mean-Field Stochastic Differential Equations with irregular drift, \emph{arXiv:1806.11451}.
%\bibitem{BFY}  A. Bensoussan, J. Frehse, P. Yam, Mean field games and mean field type control theory, \emph{Springer Briefs in Mathematics.} Springer, New York, 2013.
%\bibitem{BO} R. J. Berman, M. \"{O}nnheim, Propagation of Chaos for a Class of First Order Models with Singular Mean Field Interactions, \emph{SIAM J. Math. Anal.} 51(2019), 159-196.
\bibitem{BJW} D. Bresch, P.-E. Jabin, Z. Wang, Mean-field limit and quantitative estimates with singular attractive kernels, \emph{Duke Math. J.} 172(2023), 2591-2641.

%\bibitem{BLY} L. Bo, T. Li, X. Yu, Centralized systemic risk control in the interbank system: weak formulation and Gamma-convergence, \emph{Stochastic Process. Appl.} 150(2022), 622-654.

%\bibitem{CD} R. Carmona,  F. Delarue, Probabilistic Theory of Mean Field Games with Applications II, \emph{Mean-field Games with Common Noise and Master Equations,} Springer, 2018.
%\bibitem{CDL}  R. Carmona, F. Delarue, D. Lacker, Mean field games with common noise, \emph{Ann. Probab.} 44(2016), 3740-3803.
%\bibitem{CR} P. E. Chaudru de Raynal, Strong well-posedness of McKean-Vlasov stochastic differential equation with H\"older drift,  \emph{Stochastic Process. Appl.}  130(2020),     79--107.
%\bibitem{CF} P. E. Chaudru de Raynal, N. Frikha, Well-posedness for some non-linear diffusion processes and related pde on the wasserstein space,  \emph{arXiv:1811.06904}.
\bibitem{DEGZ} A. Durmus,  A. Eberle, A. Guillin, R. Zimmer, An elementary approach to uniform in time propagation of chaos, \emph{Proc. Amer. Math. Soc.} 148(2020), 5387-5398.
%\bibitem{CM} D. Crisan, E. McMurray, Smoothing properties of McKean-Vlasov SDEs, \emph{Probab. Theory Relat. Fields} 171(2018), 97--148.
%\bibitem{DV} D. Dawson, J. Vaillancourt. Stochastic McKean-Vlasov equations, Nonlinear Differential
%Equations. Appl. 2(1995),199--229.

%\bibitem{FKM} N. Frikha, V. Konakov, S. Menozzi, Well-posedness of some non-linear stable driven SDEs, 	\emph{arXiv:1910.05945}.
%\bibitem{FG} N. Fournier, A. Guillin, On the rate of convergence in Wasserstein distance of the empirical measure,\emph{ Probab. Theory Related Fields} 162(2015), 707-738.
\bibitem{GBM} A. Guillin, P. Le Bris, P. Monmarch\'{e}, Uniform in time propagation of chaos for the 2D vortex model and other singular stochastic systems, \emph{J. Eur. Math. Soc.} (2024), DOI 10.4171/JEMS/1413.

\bibitem{GBMEJP} A. Guillin, P. Le Bris, P. Monmarch\'{e}, Convergence rates for the Vlasov-Fokker-Planck equation and uniform in time propagation of chaos in non convex cases, \emph{Electron. J. Probab.} 27(2022), Paper No. 124, 44 pp.

\bibitem{HRW}  X. Huang,  M. R\"{o}ckner, F.-Y. Wang,  Non-linear Fokker--Planck equations for probability measures on path space and path-distribution dependent SDEs.  \emph{Discrete Contin. Dyn. Syst.} 39(2019), 3017-3035.
%\bibitem{HRW23} X. Huang, P. Ren, F.-Y. Wang, Probability Distance Estimates Between Diffusion Processes and Applications to Singular McKean-Vlasov SDEs, \emph{J. Differential Equations}  420(2025), 376-399.
%\bibitem{HXJMAA} X. Huang, F.-Y. Wang, Singular McKean-Vlasov (reflecting) SDEs with distribution dependent noise, \emph{J. Math. Anal. Appl.} 514(2022), Paper No. 126301, 21 pp.
\bibitem{HX25d} X. Huang,   Uniform in Time Propagation of Chaos for Mean Field Particle System with Interacting Noise and Partially Dissipative Drifts, \emph{arXiv:2409.01606v3}.
%\bibitem{FH} N. Fournier, M. Hauray, Propagation of Chaos for the Landau equation with moderately soft potentials, \emph{Ann. Probab.} 44(2016), 3581-3660.

%\bibitem{FLLLT} E. Fourni\'{e}, J.-M. Lasry, J. Lebuchoux, P.-L. Lions, N. Touzi, Applications of Malliavin calculus to Monte Carlo methods in finance, \emph{Finance Stoch.} 3(1999), 391-412.



%\bibitem{GLL} O. Gu\'{e}ant O, J.-M. Lasry, P.-L. Lions, Mean Field Games and Applications, \emph{Paris-Princeton Lectures on Mathematical Finance 2010. Lecture Notes in Math.} 2003, Springer, Berlin, 205-266.
%\bibitem{GLWZ} A. Guillin, W. Liu,  L. Wu, C. Zhang, The kinetic Fokker-Planck equation with mean field interaction,     \emph{J. Math. Pures Appl.} 150(2021),1-23.

%\bibitem{HU} M. Hahn, S. Umarov, Fractional Fokker-Planck-Kolmogorov type equations and their associated stochastic differential equations, \emph{Fract. Calc. Appl. Anal.\emph}  14(2011), 56每79.
%\bibitem{HLL} W. Hong, S. Li, W. Liu, Strong convergence rates in averaging principle for slow-fast McKean-Vlasov SPDEs, \emph{J. Differential Equations}  316(2022), 94-135.
%\bibitem{HX} X. Huang, Path-distribution dependent SDEs with singular coefficients, \emph{Electron. J. Probab.} 26(2021), 1-21.
%\bibitem{HL} X. Huang, W. Lv, Exponential Ergodicity and 髦蜭換畦 for Path-Distribution Dependent Stochastic Hamiltonian System, \emph{arXiv:2109.13728}.
% \bibitem{HS} X. Huang, Y. Song, Well-posedness and regularity for distribution dependent SPDEs with singular drifts, \emph{Nonlinear Anal.} 203(2021), 112167, 18 pp.
%\bibitem{HW19} X. Huang, F.-Y. Wang, Distribution dependent SDEs with singular coefficients,  \emph{Stochastic Process. Appl.} 129(2019), 4747--4770.
%\bibitem{HY21} X. Huang, F.-F. Yang, Distribution-dependent SDEs with H\"{o}lder continuous drift and $\alpha$-stable noise, \emph{Numer. Algorithms} 86(2021), 813每831.
%\bibitem{Kac}  M. Kac, Foundations of kinetic theory. In: Proceedings of the Third Berkeley Symposium on Mathematical Statistics and Probability, 1954-1955, Vol III. Berkeley: Univ of California Press, (1956), 171-197.
%\bibitem{Kac2} M. Kac, Probability and Related Topics in the Physical Sciences, Interscience, New York, 1958.
%\bibitem{K} V. N. Kolokoltsov, Nonlinear Markov processes and kinetic equations, \emph{Cambridge Tracks in Mathematics 182,} Cambridge Univ. Press, 2010.
\bibitem{JW} P.-E. Jabin,  Z. Wang, Quantitative estimates of propagation of chaos for stochastic systems with $W^{-1,\infty}$ kernels, \emph{Invent. Math.} 214(2018), 523-591.

\bibitem{JW1} P.-E. Jabin, Z. Wang, Mean field limit and propagation of chaos for Vlasov systems with bounded forces, \emph{J. Funct. Anal.} 271(2016), 3588-3627.

%\bibitem{L}  D. Lacker, Hierarchies, entropy, and quantitative propagation of chaos for mean field diffusions, \emph{arXiv:2105.02983}.

%\bibitem{L} D. Lacker, On a strong form of propagation of chaos for McKean-Vlasov equations, \emph{Electron. Commun. Probab.} 23(2018), 1-11.
%\bibitem{LL} J. M. Lasry, P. L. Lions, Mean field games, \emph{Jpn. J. Math.} 2(2007), 229-260.
%\bibitem{Li} Z. Li, Measure-Valued Branching Markov Processes, \emph{Probability and its Applications (New York),} Springer, Heidelberg, 2011.
%\bibitem{LMW} M. Liang, M. B. Majka, J. Wang, Exponential ergodicity for SDEs and McKean-Vlasov processes with L\'{e}vy noise,  \emph{Ann. Inst. Henri Poincar\'e Probab. Stat.} 57(2021), 1665-1701.
%\bibitem{MV} Yu. S. Mishura, A. Yu. Veretennikov, Existence and uniqueness theorems for solutions of McKean-Vlasov stochastic equations, \emph{arXiv:1603.02212}.

\bibitem{L21} D. Lacker, Hierarchies, entropy, and quantitative propagation of chaos for mean field diffusions, \emph{Probab. Math. Phys.} 4(2023), 377-432.

\bibitem{LL} D. Lacker, L. Le Flem, Sharp uniform-in-time propagation of chaos, \emph{Probab. Theory Related Fields} 187(2023), 443-480.
\bibitem{LWZ} W. Liu, L. Wu, C. Zhang, Long-time behaviors of mean-field interacting particle systems related to McKean-Vlasov equations, \emph{Comm. Math. Phys.} 387(2021), 179-214.

\bibitem{Luk} J. Lukkarinen, Generation and propagation of chaos in the stochastic Kac model, Talk at Rutgers(2023).

\bibitem{M}  F. Malrieu, Logarithmic sobolev inequalities for some nonlinear PDE's, \emph{Stochastic Process. Appl.} 95(2001), 109-132.

\bibitem{McKean} H. P. McKean,  A class of Markov processes associated with nonlinear parabolic equations,  \emph{Proc. Nat. Acad. Sci. U.S.A.} 56(1966), 1907-1911.
\bibitem{McKean67} H. P. McKean, Propagation of chaos for a class of non-linear parabolic equations, \emph{Stochastic Differential Equations (Lecture Series in Differential Equations, Session 7, Catholic Univ., 1967)}, pp. 41-57.
%\bibitem{23R} P. Ren, Singular McKean-Vlasov SDEs: well-posedness, regularities and Wang's Harnack inequality, \emph{Stochastic Process. Appl.} 156(2023), 291-311.
\bibitem{MRW} P. Monmarch\'{e}, Z. Ren, S. Wang, Time-uniform log-Sobolev inequalities and applications to propagation of chaos, \emph{Electron. J. Probab.}  29(2024), Paper No. 154, 38 pp.
\bibitem{RS} M. Rosenzweig, S. Serfaty, Modulated logarithmic Sobolev inequalities and generation of chaos. \emph{Ann. Fac. Sci. Toulouse Math. } 34(2025) 107-134.

\bibitem{23RW} P. Ren, F.-Y. Wang, Entropy Estimate Between Diffusion Processes with Application to Nonlinear Fokker-Planck Equations, \emph{arXiv:2302.13500v2}.
%\bibitem{MD} L. Miclo, P. Del Moral,
%Genealogies and Increasing Propagation of Chaos For Feynman-Kac and Genetic Models,
%\emph{Ann. Appl. Probab.} 11(2001), 1166-1198.
%\bibitem{S} T. Seidman,\emph{ How violent are fast controls?}  Math. Control Signals Systems  1(1988), 89-95.
%\bibitem{RW18} P. Ren, F.-Y. Wang, Bismut formula for Lions derivative of distribution dependent SDEs and applications, \emph{J. Differential Equations} 267(2019), 4745--4777.
% \bibitem{RW20b} P. Ren, F.-Y. Wang, Exponential convergence in entropy and Wasserstein distance for McKean-Vlasov SDEs, \emph{Nonlinear Anal.} 206(2021), 112259.
%\bibitem{RZ} M. R\"ockner, X. Zhang, Well-posedness of distribution dependent SDEs with singular drifts, \emph{Bernoulli} 27(2021), 1131-1158.
%\bibitem{Song}  Y. Song, Gradient estimates and exponential ergodicity for mean-field SDEs, \emph{J. Theort. Probab.} 33(2020), 201--238.
%\bibitem{Szn} A. S. Sznitman,  Topics in propagations of chaos, \emph{In: Hennequin PL. (eds) Ecole d'Et\'{e} de Probabilit\'{e}s de Saint-Flour XIX每1989. Lecture Notes in Math,} Vol 1464. Berlin: Springer, (1991), 165每251.
%\bibitem{SW} J. Shao, D. Wei, Propagation of chaos and conditional McKean-Vlasov SDEs with regime-switching. \emph{Front. Math. China} 17(2022), 731-746.

\bibitem{SZ} A.-S. Sznitman,   \emph{Topics in propagation of chaos,} In $``$\'Ecole d'\'Et\'e de Probabilit\'es de Sain-Flour XIX-1989", Lecture Notes in Mathematics  1464, p. 165-251, Springer, Berlin, 1991.


\bibitem{Wbook} F.-Y. Wang, \emph{Harnack Inequality for Stochastic Partial Differential Equations,} Springer, New York, 2013.
\bibitem{W18} F.-Y. Wang, Distribution dependent SDEs for Landau type equations, \emph{Stochastic Process. Appl.} 128(2018), 595-621.
%\bibitem{W22} F.-Y. Wang, Distribution Dependent Reflecting Stochastic Differential Equations, \emph{Sci. China Math.} 66(2023), 2411-2456.
%\bibitem{WHGY} H. Wu, J. Hu, S. Gao, C. Yuan, Stabilization of Stochastic McKean--Vlasov Equations with Feedback Control Based on Discrete-Time State Observation,
%\bibitem{W21a} F.-Y. Wang, Derivative Formula for Singular McKean-Vlasov SDEs, \emph{Commun. Pure Appl. Anal.} 22(2023), 1866-1898.
%\bibitem{XXZZ} P. Xia, L. Xie, X. Zhang, G. Zhao, \emph{$L^q$($L^p$)-theory of stochastic differential equations,} Stochatic Process. Appl. 130(2020), 5188-5211.
%\bibitem{XZ} L. Xie, X. Zhang, \emph{ Ergodicity of stochastic differential equations with jumps and singular coefficients,}  Ann. Inst. Henri Poincar\'e Probab. Stat. 56(2020),  175-229.
%\bibitem{W21f} F.-Y. Wang, Distribution Dependent Reflecting Stochastic Differential Equations, \emph{arXiv:2106.12737}.
%\bibitem{YZ}  C. Yuan, S.-Q. Zhang,  A Zvonkin's transformation for stochastic differential equations with singular drift and applications, \emph{J. Differential Equations} 297(2021), 277-319.

\end{thebibliography}
\end{document}